\documentclass[12pt,reqno]{amsart}
\usepackage[margin=1in]{geometry}
\usepackage{amsmath,amssymb,amsthm,graphicx,amsxtra, setspace}
\usepackage[utf8]{inputenc}
\usepackage{mathrsfs}
\usepackage{hyperref}
\usepackage{upgreek}
\usepackage{mathtools}
\usepackage[dvipsnames]{xcolor}
\usepackage[mathcal]{euscript}
\usepackage{amsmath,units}
\usepackage{pgfplots}
\usepackage{tikz}
\usepackage{verbatim}
\allowdisplaybreaks
\usepackage{dsfont}
\usepackage{fontenc}
\usepackage{textcomp}
\usepackage{marvosym}
\usepackage{eurosym}
\usepackage{upgreek}
\usepackage[pagewise]{lineno}

\DeclareMathAlphabet{\mathpzc}{OT1}{pzc}{m}{it}

\usepackage[cyr]{aeguill}

\colorlet{darkblue}{blue!50!black}

\hypersetup{
	colorlinks,%
	citecolor=blue,%
	filecolor=red,%
	linkcolor=red,%
	urlcolor=blue,%
	pdfnewwindow=true,%
	pdfstartview={FitH}
}

\newtheorem{theorem}{Theorem}[section]
\newtheorem{lemma}[theorem]{Lemma}
\newtheorem{proposition}[theorem]{Proposition}

\newtheorem{definition}[theorem]{Definition}

\newtheorem{remark}[theorem]{Remark}

\let\originalleft\left
\let\originalright\right
\renewcommand{\left}{\mathopen{}\mathclose\bgroup\originalleft}
\renewcommand{\right}{\aftergroup\egroup\originalright}

\renewcommand{\d}{\/\mathrm{d}\/}

\def\w{\textbf{W}^{\varepsilon}_{{\theta}^{\varepsilon}}}

\def\var{\varepsilon}

\def\L{\mathbb{L}}
\def\A{\mathrm{A}}
\def\I{\mathrm{I}}

\def\C{\mathrm{C}}
\def\f{\boldsymbol{f}}

\def\B{\mathrm{B}}
\def\D{\mathrm{D}}
\def\y{\boldsymbol{y}}

\def\X{\mathbb{X}}
\def\x{\boldsymbol{x}}

\def\z{\boldsymbol{z}}
\def\v{\boldsymbol{v}}
\def\w{\boldsymbol{w}}
\def\W{\mathrm{W}}
\def\G{\mathrm{G}}

\def\N{\mathbb{N}}

\def\V{\mathbb{V}}
\def\wi{\widetilde}

\def\u{\mathrm{U}}
\def\P{\mathrm{P}}
\def\u{\boldsymbol{u}}
\def\H{\mathbb{H}}
\def\n{\boldsymbol{n}}

\newcommand{\R}{\mathbb{R}}

\renewcommand{\d}{\/\mathrm{d}\/}

\newcommand{\Addresses}{{
		\footnote{
			
			\noindent \textsuperscript{1}Department of Mathematics, Indian Institute of Technology Roorkee-IIT Roorkee,
			Haridwar Highway, Roorkee, Uttarakhand 247667, INDIA.\par\nopagebreak
			\noindent  \textit{e-mail:} \texttt{maniltmohan@ma.iitr.ac.in, maniltmohan@gmail.com.}
			
			\noindent \textsuperscript{*}Corresponding author.

			\textit{Key words:} Brinkman--Forchheimer equations, weak solution, strong solution, monotonicity. 
			
			Mathematics Subject Classification (2010): 37L30, 35Q35, 35Q30, 35B40.

}}}

\begin{document}
	
	
	\title[On the convective Brinkman--Forchheimer equations]{On the convective Brinkman--Forchheimer equations
			\Addresses}
	\author[S. Gautam and M. T. Mohan ]{Sagar Gautam\textsuperscript{1} and Manil T. Mohan\textsuperscript{1*}}

	\maketitle
	
	\begin{abstract}
	The  convective Brinkman--Forchheimer  equations or the Navier--Stokes equations with damping in  bounded or  periodic domains $\subset\mathbb{R}^d$, $2\leq d\leq 4$ are considered in this work.	The existence and uniqueness of a global weak solution in the Leray-Hopf sense satisfying the energy equality to the system: 	$$\partial_t\boldsymbol{u}-\mu \Delta\boldsymbol{u}+(\boldsymbol{u}\cdot\nabla)\boldsymbol{u}+\alpha\boldsymbol{u}+\beta|\boldsymbol{u}|^{r-1}\boldsymbol{u}+\nabla p=\boldsymbol{f},\ \nabla\cdot\boldsymbol{u}=0,$$	 (for all values of $\beta>0$ and $\mu>0$, whenever the absorption exponent $r>3$ and  $2\beta\mu \geq 1$, for the critical case $r=3$) is proved. We exploit the monotonicity as well as the demicontinuity properties of the linear and nonlinear operators and the Minty-Browder technique in the proofs.  Finally, we discuss the existence of global-in-time strong solutions to such systems in periodic domains. 
	\end{abstract}

	\section{Introduction}\label{sec1}\setcounter{equation}{0}
	The existence and uniqueness of global in time strong solutions of the classical  3D incompressible  Navier--Stokes equations (NSE) 
	$$\partial_t\u-\nu \Delta\u+(\u\cdot\nabla)\u+\nabla p=\f,\ \nabla\cdot\u=0,$$
	 is one of the biggest open problems in Mathematics (see \cite{FMRT,GGP,OAL,JCR3,Te,Te1}, etc.). In the recent years, several mathematicians came forward with some modifications of the classical 3D NSE and posed problems on global solvability of such models.  The Cauchy problem for the Navier--Stokes equations with damping $r|\u|^{r-1}\u$ in the whole space is considered in the work \cite{ZCQJ}, the authors showed that it has global weak solutions, for any $r\geq 1$. The existence and uniqueness of a smooth solution to a tamed 3D NSE in the whole space is established in the work  \cite{MRXZ}. Moreover, the authors  showed that if there exists a bounded smooth solution to the classical 3D NSE, then this solution satisfies the tamed equation.  Whereas, the bounded domains are concerned, the authors in \cite{SNA} considered the Navier--Stokes problem in bounded domains with compact boundary, modified by the absorption term $|\u|^{r-2}\u$, for $r>2$. For this modified problem, they proved the existence of weak solutions in the Leray-Hopf sense, for any dimension $d\geq 2$ and their uniqueness only for $d=2$. Furthermore,  in three dimensions, they were not able to establish the energy equality satisfied by the weak solutions and the authors in \cite{KWH,CLF} resolved this issue recently.

This work is concerned about the convective Brinkman--Forchheimer equations in bounded or periodic domains. The model described below is formulated on bounded domains, but one can formulate the same problem in periodic domains also (cf. \cite{KWH}). Let $\mathcal{O}\subset\R^d$ ($2\leq d\leq 4$) be a bounded domain with a smooth boundary $\partial\mathcal{O}$ (at least $\C^2-$boundary). The	convective Brinkman--Forchheimer (CBF)  equations are given by (see \cite{KT2} for Brinkman--Forchheimer equations with fast growing nonlinearities)
\begin{equation}\label{1}
\left\{
\begin{aligned}
\frac{\partial \u}{\partial t}-\mu \Delta\u+(\u\cdot\nabla)\u+\alpha\u+\beta|\u|^{r-1}\u+\nabla p&=\boldsymbol{f}, \ \text{ in } \ \mathcal{O}\times(0,T), \\ \nabla\cdot\u&=0, \ \text{ in } \ \mathcal{O}\times(0,T), \\
\u&=\mathbf{0}\ \text{ on } \ \partial\mathcal{O}\times(0,T), \\
\u(0)&=\u_0 \ \text{ in } \ \mathcal{O}.
\end{aligned}
\right.
\end{equation}
The convective Brinkman--Forchheimer equations \eqref{1} describe the motion of incompressible fluid flows in a saturated porous medium.  In the above system, $\u(x,t)\in\R^d$ represents the velocity field at time $t$ and position $x$, $p(x,t)\in\R$ denotes the pressure, $\f(x,t)\in\R^d$ is an external forcing. For the uniqueness of the pressure $p$, one can impose the condition $\int_{\mathcal{O}}p(x,t)\d x=0 \text{ in }  (0,T)$. The constant $\mu$ represents the positive Brinkman coefficient (effective viscosity), the positive constants $\alpha$ and $\beta$ represent the Darcy (permeability of porous medium) and Forchheimer (proportional to the porosity of the material) coefficients, respectively. The absorption exponent $r\in[1,\infty)$ and the exponent $r=3$ is known as the critical exponent.  It can be easily seen that for $\alpha=\beta=0$, we obtain the classical 3D Navier--Stokes equations. Our aim in this article is to prove the existence and uniqueness of global weak as well as strong solutions to the system \eqref{1} for $r\geq 3$ in bounded or periodic domains. We establish a monotonicity property of the linear and nonlinear operators, and make use of the well-known Minty-Browder technique to obtain the existence of global weak solutions in the Leray-Hopf sense satisfying the energy equality  
\begin{align}
	\label{1.2}\u\in\C([0,T];\H)\cap\mathrm{L}^2(0,T;\V)\cap\mathrm{L}^{r+1}(0,T;\wi{\L}^{r+1}),
\end{align}
to the system \eqref{1}, for $\u_0\in\H$ and $\f\in\mathrm{L}^2(0,T;\V')$ (see section \ref{sec2} for function spaces).  Stochastic versions of such methods are well-known in the solvability of stochastic partial differential equations arising in Mathematical Physics (see \cite{ICAM,MJSS,MTM}, etc. and references therein) and this is our main motivation to devise such a method for CBF equations. The stochastic counterpart of the problem \eqref{1} has been already done in \cite{MT4} (see also \cite{MT2,MTM11,MTM12,MTM14}). In periodic domains, we obtain the following regularity of global strong solutions  
\begin{align}\label{1.3}
	\u\in\C([0,T];\V)\cap\mathrm{L}^{\infty}(0,T;\wi{\L}^{r+1})\cap\mathrm{L}^2(0,T;\D(\A))\cap\mathrm{L}^{r+1}(0,T;\wi{\L}^{3(r+1)}),
\end{align} 
for $\u_0\in\V\cap\wi{\L}^{r+1}$ and $\f\in\mathrm{L}^2(0,T;\H)$. For $\u_0\in\D(\A)$ and $\f\in\mathrm{W}^{1,1}(0,T;\H)$, we establish that the  system \eqref{1} possesses unique strong solution with 
\begin{align}\label{1.4}
	\u\in\mathrm{W}^{1,\infty}([0,T];\H)\cap\mathrm{L}^{\infty}(0,T;\D(\A)). 
\end{align}
We made use of the abstract theory available in \cite{VB1,VB}, etc. to arrive at such a result. 
 In bounded domains, due to technical difficulties described below, we are not able to obtain the regularity of weak solutions in the above class. 

Now, we discuss some of the global solvability results available in the literature for the 3D CBF equations and related models in the whole space as well as in periodic domains. The Cauchy problem corresponding to \eqref{1} in the whole space (with $\alpha=0$ and $\beta=r$) is considered in \cite{ZCQJ}. The authors showed that the system has global weak solutions, for any $r\geq 1$, global strong solutions, for any $r\geq 7/2$ and that the strong solution is unique, for any $7/2\leq r\leq 5$. An improvement of this result was made in \cite{ZZXW}. The authors showed that the above-mentioned problem possesses global strong solutions, for any $r>3$ and the strong solution is unique, when $3<r\leq 5$. Later, the authors in \cite{YZ}  proved that the strong solution exists globally for $r\geq 3$, and they established two regularity criteria, for $1\leq r<3$. Moreover, for any $r\geq 1$, they proved that the strong solution is unique even among weak solutions.  In \cite{KWH}, the authors obtained a simple proof of the existence of global-in-time smooth solutions for the CBF equations \eqref{1} with $r>3$ on a 3D periodic domain. For the critical value $r=3$, they proved that the unique global, regular solutions exist, provided that the coefficients satisfy the relation $4\beta\mu\geq 1$. Furthermore, they resolved the issue of establishing the energy equality satisfied by the weak solutions  in the critical case. On the 3D torus, they were able to approximate functions in $\L^p$-spaces using truncated Fourier expansions to get such a result. Recently, the authors in \cite{KWH1} showed that the strong solutions of 3D CBF  equations in periodic domains with  the absorption exponent $r\in[1,3]$ remain strong under small enough changes of initial condition and forcing function. Moreover, they also established the weak-strong uniqueness of the torus.

In the bounded domains, the existence of a global weak solution to the system \eqref{1} is established in \cite{SNA}. Unlike the whole space or periodic domains, in bounded domains, there is a technical difficulty for obtaining strong solutions to \eqref{1} with the regularity given in \eqref{1.3} for the velocity field $\u(\cdot)$.  As mentioned  in the paper \cite{KT2}, the major difficulty in working with bounded domains is that $\mathcal{P}(|\u|^{r-1}\u)$ (here $\mathcal{P}:\L^2(\mathcal{O})\to\H$ is the Helmholtz-Hodge orthogonal projection) need not be zero on the boundary, and $\mathcal{P}$ and $-\Delta$ are not necessarily commuting (for a counter example, see \cite[Example 2.19, Chapter 2, pp. 57]{JCR4}). Moreover, $\Delta\u\big|_{\partial\mathcal{O}}\neq 0$ in general and the term with pressure will not disappear (see \cite{KT2}), while taking the inner product of the first equation in \eqref{1} with $\Delta\u$. Therefore, the equality (\cite{KWH})
\begin{align}\label{3}
&\int_{\mathcal{O}}(-\Delta\u(x))\cdot|\u(x)|^{r-1}\u(x)\d x\nonumber\\&=\int_{\mathcal{O}}|\nabla\u(x)|^2|\u(x)|^{r-1}\d x+4\left[\frac{r-1}{(r+1)^2}\right]\int_{\mathcal{O}}|\nabla|\u(x)|^{\frac{r+1}{2}}|^2\d x\nonumber\\&=\int_{\mathcal{O}}|\nabla\u(x)|^2|\u(x)|^{r-1}\d x+\frac{r-1}{4}\int_{\mathcal{O}}|\u(x)|^{r-3}|\nabla|\u(x)|^2|^2\d x,
\end{align}
may not be useful in the context of bounded domains. The authors in \cite{KT2} considered the Brinkman--Forchheimer equations with fast growing nonlinearities and they showed the existence of regular dissipative solutions and global attractors for the system \eqref{1} with $r> 3$. The existence of a global smooth solution assures that the energy equality is satisfied by the weak solutions in bounded domains. But the critical case of $r=3$ remained open until it is resolved by the authors in \cite{CLF}. They were able to construct functions that can approximate functions defined on smooth bounded domains by elements of eigenspaces of the Stokes operator in such a way that the approximations are bounded and converge in both Sobolev and Lebesgue spaces simultaneously. As a simple application of this result, they proved that all weak solutions of the critical CBF equations ($r=3$) posed on a bounded domain in $\mathbb{R}^3$ satisfy the energy equality. 

The rest of the paper is organized as follows. In the next section, we discuss the functional setting of the problem described in \eqref{1}. After defining necessary function spaces, we define the linear and nonlinear operators and show that these operators satisfy a monotonicity property for $r\geq 3$ (see Theorems \ref{thm2.2} and \ref{thm2.3}). Furthermore, we show that the sum of linear and nonlinear operators satisfy demicontinuous property also (Lemma \ref{lem2.8}). After providing an abstract formulation of the system \eqref{1}, the existence and uniqueness of a global weak solution in the Leray-Hopf sense is examined in Section \ref{sec3}. The monotonicity and hemicontinuity properties of the linear and nonlinear operators as well as the Minty-Browder techniques are exploited in the proofs (Theorem \ref{main}). In the final section, we discuss the global strong solutions to the system \eqref{1}. Due to technical difficulties described above, we are able to prove the regularity of the weak solutions in the class \eqref{1.3} in periodic domains only (Theorem \ref{reg}). Moreover, using the abstract theory developed in \cite{VB1,VB}, we prove the existence and uniqueness of global strong solutions for the system \eqref{1} in the class \eqref{1.3} (Theorem \ref{thm4.2}).

\section{Mathematical Formulation}\label{sec2}\setcounter{equation}{0}
In this section, we provide the necessary function spaces needed to obtain the global solvability results of the system \eqref{1}. After defining linear and nonlinear operators, we prove the properties of these operators like monotonicity and  hemicontinuity.   In our analysis, the parameter $\alpha$ does not play a major role and we set $\alpha$ to be zero in \eqref{1} in the entire paper.

\subsection{Function spaces} Let $\C_0^{\infty}(\mathcal{O};\R^d)$ be the space of all infinitely differentiable functions  ($\R^d$-valued) with compact support in $\mathcal{O}\subset\R^d$.  Let us define 
\begin{align*} 
\mathcal{V}&:=\{\u\in\C_0^{\infty}(\mathcal{O},\R^d):\nabla\cdot\u=0\},\\
\mathbb{H}&:=\text{the closure of }\ \mathcal{V} \ \text{ in the Lebesgue space } \L^2(\mathcal{O})=\mathrm{L}^2(\mathcal{O};\R^d),\\
\mathbb{V}&:=\text{the closure of }\ \mathcal{V} \ \text{ in the Sobolev space } \H_0^1(\mathcal{O})=\mathrm{H}_0^1(\mathcal{O};\R^d),\\
\widetilde{\L}^{p}&:=\text{the closure of }\ \mathcal{V} \ \text{ in the Lebesgue space } \L^p(\mathcal{O})=\mathrm{L}^p(\mathcal{O};\R^d),
\end{align*}
for $p\in(2,\infty)$. The space of divergence-free test functions in the space-time domain is defined by 
\begin{align}\label{dvtest}
	\mathcal{V}_T:=\left\{\v\in\C_0^{\infty}(\mathcal{O}\times[0,T);\R^d):\nabla\cdot\v(\cdot,t)=0\right\}.
\end{align} 
It should be noted that $\v(x,T)=0$, for all $\v\in\mathcal{V}_T$. Then under some smoothness assumptions on the boundary, we characterize the spaces $\H$, $\V$ and $\widetilde{\L}^p$ as 
$
\H=\{\u\in\L^2(\mathcal{O}):\nabla\cdot\u=0,\u\cdot\mathbf{n}\big|_{\partial\mathcal{O}}=0\}$,  with norm  $\|\u\|_{\H}^2:=\int_{\mathcal{O}}|\u(x)|^2\d x,
$
where $\mathbf{n}$ is the outward normal to $\partial\mathcal{O}$,
$
\V=\{\u\in\H_0^1(\mathcal{O}):\nabla\cdot\u=0\},$  with norm $ \|\u\|_{\V}^2:=\int_{\mathcal{O}}|\nabla\u(x)|^2\d x,
$ and $\widetilde{\L}^p=\{\u\in\L^p(\mathcal{O}):\nabla\cdot\u=0, \u\cdot\mathbf{n}\big|_{\partial\mathcal{O}}=0\},$ with norm $\|\u\|_{\widetilde{\L}^p}^p=\int_{\mathcal{O}}|\u(x)|^p\d x$, respectively.
Let $(\cdot,\cdot)$ denote the inner product in the Hilbert space $\H$ and $\langle \cdot,\cdot\rangle $ denotes the induced duality between the spaces $\V$  and its dual $\V'$ as well as $\widetilde{\L}^p$ and its dual $\widetilde{\L}^{p'}$, where $\frac{1}{p}+\frac{1}{p'}=1$. Note that $\H$ can be identified with its dual $\H'$. From \cite[Subsection 2.1]{FKS}, we have that the sum space $\V'+\widetilde{\L}^{p'}$ is well defined and  is a Banach space with respect to the norm 
\begin{align}\label{22}
	\|\u\|_{\V'+\widetilde{\L}^{p'}}&:=\inf\{\|\u_1\|_{\V'}+\|\u_2\|_{\wi\L^{p'}}:\u=\u_1+\u_2, \y_1\in\V' \ \text{and} \ \y_2\in\wi\L^{p'}\}\nonumber\\&=
	\sup\left\{\frac{|\langle\u_1+\u_2,\f\rangle|}{\|\f\|_{\V\cap\widetilde{\L}^p}}:\boldsymbol{0}\neq\f\in\V\cap\widetilde{\L}^p\right\},
\end{align}
where $\|\cdot\|_{\V\cap\widetilde{\L}^p}:=\max\{\|\cdot\|_{\V}, \|\cdot\|_{\wi\L^p}\}$ is a norm on the Banach space $\V\cap\widetilde{\L}^p$. Also the norm $\max\{\|\u\|_{\V}, \|\u\|_{\wi\L^p}\}$ is equivalent to the norms  $\|\u\|_{\V}+\|\u\|_{\widetilde{\L}^{p}}$ and $\sqrt{\|\u\|_{\V}^2+\|\u\|_{\widetilde{\L}^{p}}^2}$ on the space $\V\cap\widetilde{\L}^p$. Moreover, we have the continuous embeddings $$\V\cap\widetilde{\L}^p\hookrightarrow\V\hookrightarrow\H\cong\H'\hookrightarrow\V'\hookrightarrow\V'+\widetilde{\L}^{p'},$$ where the embedding $\V\hookrightarrow\H$ is compact. We denote $\H^2(\mathcal{O})=\W^{2,2}(\mathcal{O};\R^d)$ for the second order Hilbertian-Sobolev spaces. For the functional set up in periodic domains, interested readers are referred to see \cite{KWH,Te1}, etc. The following Gr\"onwall inequality is used in the sequel and has been taken from \cite{JRRY}.
\begin{lemma}[Gr\"onwall inequality]\label{Gronw}
	Suppose ${y},{f},{f}_1$ and ${f}_2$ are four non-negative locally integrable functions on $[t_0,\infty)$, $t_0>0$, satisfying
	\begin{align*}
		{y}(t)+\int_{t_0}^t{f}(s)\d s\leq {a}+\int_{t_0}^t {f}_1(s)\d s+\int_{t_0}^{t} {f}_2(s){y}(s)\d s, \ \text{ for all } \ t\in[t_0,\infty),
	\end{align*} 
	where ${a}$ is some non-negative constant. Then, 
	\begin{align*}
		{y}(t)+\int_{t_0}^t{f}(s)\d s\leq 
		\left({a}+\int_{t_0}^t {f}_1(s)\d s\right)\exp\left(\int_{t_0}^{t} {f}_2(s)\d s\right),
	\end{align*}
	for all $t\in[t_0,\infty)$.
\end{lemma}

The following Gr\"onwall inequality is a nonlinear generalization of Lemma \ref{Gronw} and has been taken from \cite[Theorem 21, Chapter 1, pp. 11]{SSD}.
\begin{lemma}[Gr\"onwall inequality: A nonlinear generalization]\label{lem-non-gro}
	Let $y$ be a non-negative function that satisfies the following  integral inequality:
	\begin{align}
     y(t)\leq c+\int_{t_0}^t(a(s)y(s)+b(s)y^{\alpha}(s))\d s,  \ \text{ for all } \ t\in[t_0,\infty), 
		\end{align}
		where $c$ is some non-negative constant, $\ \alpha\in[0,1)$ and, $a(t)$ and $b(t)$ are  non-negative locally integrable  functions in $[t_0,\infty)$. Then, we have 
		\begin{align}
			y(t)\leq & \left\{c^{1-\alpha}\exp\left[(1-\alpha)\int_{t_0}^ta(s)\d s\right]\right.\nonumber\\&\left.\quad+(1-\alpha)\int_{t_0}^tb(s)\exp\left[(1-\alpha)\int_s^ta(r)\d r\right]\d s\right\}^{\frac{1}{1-\alpha}},
		\end{align}
		for all $t\in[t_0,\infty)$.
\end{lemma}

\subsection{Linear operator}\label{opeA}
It is well-known from \cite{DFHM,HKTY} that every vector field $\u\in\mathbb{L}^p(\mathcal{O})$, for $1<p<\infty$ can be uniquely represented as $\u=\v+\nabla q,$ where $\v\in\mathbb{L}^p(\mathcal{O})$ with $\mathrm{div \ }\v=0$ in the sense of distributions in $\mathcal{O}$ with $\v\cdot\n=0$ on $\partial\mathcal{O}$ (in the sense of trace) and $q\in\mathrm{W}^{1,p}(\mathcal{O})$ (Helmholtz-Weyl or Helmholtz-Hodge decomposition). For smooth vector fields in $\mathcal{O}$, such a decomposition is an orthogonal sum in $\mathbb{L}^2(\mathcal{O})$. Note that $\u=\v+\nabla q$ holds for all $\u\in\mathbb{L}^p(\mathcal{O})$, so that we can define the projection operator $\mathcal{P}_p$  by $\mathcal{P}_p\u = \v$. Let us consider the set $\mathcal{E}_{p}(\mathcal{O}):=\left\{\nabla q:q\in \mathrm{W}^{1,p}({\mathcal{O}})\right\}$ equipped with the norm $\|\nabla q\|_{\L^p}$. Then, from the above discussion, we obtain $\L^p(\mathcal{O})=\wi\L^p(\mathcal{O})\oplus\mathcal{E}_{p}(\mathcal{O})$.   From  \cite[Theorem 1.4]{CSHS}, we further have 
	\begin{align*}
		\|\nabla q\|_{\L^p}\leq C\|\u\|_{\L^p}, \ \|\v\|_{\L^q}\leq (C+1)\|\u\|_{\L^p}\ \text{ and }\ \|\nabla q\|_{\L^p}+\|\v\|_{\L^q}\leq (2C+1)\|\u\|_{\L^p},
	\end{align*}
	where $C=C(\mathcal{O},p)>0$ is a constant such that \begin{align}\label{23}
		\|\nabla q\|_{\L^p}\leq C\sup_{0\neq\nabla\varphi\in \mathcal{E}_{p'}(\mathcal{O})}\frac{|\langle\nabla q,\nabla\varphi\rangle|}{\|\nabla\varphi\|_{\L^{p'}}}, \ \text{ for all }\ \nabla q\in \mathcal{E}_{p}(\mathcal{O}),
	\end{align} 
	with $\frac{1}{p}+\frac{1}{p'}=1$. Setting $\mathcal{P}_p\u:=\v$, we obtain a bounded linear operator $\mathcal{P}_p:\L^p(\mathcal{O})\to\wi\L^{p}(\mathcal{O})$ such that $\mathcal{P}_p^2=\mathcal{P}_p$ (projection). For $p=2$, $\mathcal{P}:=\mathcal{P}_2:\L^2(\mathcal{O})\to\H$ is an orthogonal projection.   Since $\mathcal{O}$ is of class $\C^2$, from \cite[Remark 1.6, Chapter 1, pp. 18]{Te}, we also infer that $\mathcal{P}$ maps $\H^1(\mathcal{O})$ into itself and is continuous for the norm of $\H^1(\mathcal{O})$.	
	We define
\begin{equation*}
\left\{
\begin{aligned}
\A\u:&=-\mathcal{P}\Delta\u,\;\u\in\D(\A),\\ \D(\A):&=\V\cap\H^{2}(\mathcal{O}).
\end{aligned}
\right.
\end{equation*}
It can be easily seen that the operator $\A$ is a non-negative self-adjoint operator in $\H$ with $\V=\D(\A^{1/2})$ and \begin{align}\label{2.7a}\langle \A\u,\u\rangle =\|\u\|_{\V}^2,\ \textrm{ for all }\ \u\in\V, \ \text{ so that }\ \|\A\u\|_{\V'}\leq \|\u\|_{\V}.\end{align}
For the bounded domain $\mathcal{O}$, the operator $\A$ is invertible and its inverse $\A^{-1}$ is bounded, self-adjoint and compact in $\H$. Thus, using spectral theorem, the spectrum of $\A$ consists of an infinite sequence $0< \lambda_1\leq \lambda_2\leq\ldots\leq \lambda_k\leq \ldots,$ with $\lambda_k\to\infty$ as $k\to\infty$ of eigenvalues. 
Moreover, there exists an orthogonal basis $\{\boldsymbol{w}_k\}_{k=1}^{\infty} $ of $\H$ consisting of eigen functions of $\A$ such that $\A \boldsymbol{w}_k =\lambda_k\boldsymbol{w}_k$,  for all $ k\in\mathbb{N}$.  We know that $\u$ can be expressed as $\u=\sum_{k=1}^{\infty}\langle\u,\boldsymbol{w}_k\rangle \boldsymbol{w}_k$ and $\A\u=\sum_{k=1}^{\infty}\lambda_k\langle\u,\boldsymbol{w}_k\rangle \boldsymbol{w}_k$. Thus, it is immediate that 
\begin{align}\label{poin}
\|\nabla\u\|_{\mathbb{H}}^2=\langle \A\u,\u\rangle =\sum_{k=1}^{\infty}\lambda_k|\langle \u,\boldsymbol{w}_k\rangle|^2\geq \lambda_1\sum_{k=1}^{\infty}|\langle\u,\boldsymbol{w}_k\rangle|^2=\lambda_1\|\u\|_{\mathbb{H}}^2.
\end{align}

It should be noted that, while proving global weak solutions, we are not using the Gagliardo-Nirenberg, Ladyzhenskaya or Agmon inequalities. The results obtained in this work are true for $2\leq d\leq 4$ in bounded domains (for more details see step (iii) of the proof of the Theorem \ref{main}) and $d\geq 2$ in periodic domains (see Remark \ref{rem3.7}). The  following  interpolation inequality is used frequently in the paper.
Assume $1\leq s\leq r\leq t\leq \infty$, $\theta\in(0,1)$ such that $\frac{1}{r}=\frac{\theta}{s}+\frac{1-\theta}{t}$ and $\u\in\L^s(\mathcal{O})\cap\L^t(\mathcal{O})$, then we have 
\begin{align}\label{211}
\|\u\|_{\L^r}\leq\|\u\|_{\L^s}^{\theta}\|\u\|_{\L^t}^{1-\theta}. 
\end{align}

\subsection{Bilinear operator}\label{opeB}
Let us define the \emph{trilinear form} $b(\cdot,\cdot,\cdot):\V\times\V\times\V\to\R$ by $$b(\u,\v,\w)=\int_{\mathcal{O}}(\u(x)\cdot\nabla)\v(x)\cdot\w(x)\d x=\sum_{i,j=1}^d\int_{\mathcal{O}}u_i(x)\frac{\partial v_j(x)}{\partial x_i}w_j(x)\d x.$$ If $\u, \v$ are such that the linear map $b(\u, \v, \cdot) $ is continuous on $\V$, the corresponding element of $\V'$ is denoted by $\B(\u, \v)$. We also denote (with an abuse of notation) $\B(\u) = \B(\u, \u)=\mathcal{P}(\u\cdot\nabla)\u$.
An integration by parts gives 
\begin{equation}\label{b0}
\left\{
\begin{aligned}
b(\u,\v,\v) &= 0,\text{ for all }\u,\v \in\V,\\
b(\u,\v,\w) &=  -b(\u,\w,\v),\text{ for all }\u,\v,\w\in \V.
\end{aligned}
\right.\end{equation}
In the trilinear form, an application of H\"older's inequality yields
\begin{align*}
|b(\u,\v,\w)|=|b(\u,\w,\v)|\leq \|\u\|_{\widetilde{\L}^{r+1}}\|\v\|_{\widetilde{\L}^{\frac{2(r+1)}{r-1}}}\|\w\|_{\V},
\end{align*}
for all $\u\in\V\cap\widetilde{\L}^{r+1}$, $\v\in\V\cap\widetilde{\L}^{\frac{2(r+1)}{r-1}}$ and $\w\in\V$, so that we get 
\begin{align}\label{2p9}
\|\B(\u,\v)\|_{\V'}\leq \|\u\|_{\widetilde{\L}^{r+1}}\|\v\|_{\widetilde{\L}^{\frac{2(r+1)}{r-1}}}.
\end{align}
Hence, the trilinear map $b : \V\times\V\times\V\to \R$ has a unique extension to a bounded trilinear map from $(\V\cap\widetilde{\L}^{r+1})\times(\V\cap\widetilde{\L}^{\frac{2(r+1)}{r-1}})\times\V$ to $\R$. It can also be seen that $\B$ maps $ \V\cap\widetilde{\L}^{r+1}$  into $\V'+\widetilde{\L}^{\frac{r+1}{r}}$ and using interpolation inequality (see \eqref{211}), we get 
\begin{align}\label{212}
\left|\langle \B(\u,\u),\v\rangle \right|=\left|b(\u,\v,\u)\right|\leq \|\u\|_{\widetilde{\L}^{r+1}}\|\u\|_{\widetilde{\L}^{\frac{2(r+1)}{r-1}}}\|\v\|_{\V}\leq\|\u\|_{\widetilde{\L}^{r+1}}^{\frac{r+1}{r-1}}\|\u\|_{\H}^{\frac{r-3}{r-1}}\|\v\|_{\V},
\end{align}
for all $\v\in\V\cap\widetilde{\L}^{r+1}$. Thus, we have 
\begin{align}\label{2.9a}
\|\B(\u)\|_{\V'+\widetilde{\L}^{\frac{r+1}{r}}}\leq\|\u\|_{\widetilde{\L}^{r+1}}^{\frac{r+1}{r-1}}\|\u\|_{\H}^{\frac{r-3}{r-1}}.
\end{align}
Using \eqref{2p9}, for $\u,\v\in\V\cap\widetilde{\L}^{r+1}$, we also have 
\begin{align}\label{lip}
\|\B(\u)-\B(\v)\|_{\V'+\widetilde{\L}^{\frac{r+1}{r}}}&\leq \|\B(\u-\v,\u)\|_{\V'}+\|\B(\v,\u-\v)\|_{\V'}\nonumber\\&\leq \left(\|\u\|_{\widetilde{\L}^{\frac{2(r+1)}{r-1}}}+\|\v\|_{\widetilde{\L}^{\frac{2(r+1)}{r-1}}}\right)\|\u-\v\|_{\widetilde{\L}^{r+1}}\nonumber\\&\leq \left(\|\u\|_{\H}^{\frac{r-3}{r-1}}\|\u\|_{\widetilde{\L}^{r+1}}^{\frac{2}{r-1}}+\|\v\|_{\H}^{\frac{r-3}{r-1}}\|\v\|_{\widetilde{\L}^{r+1}}^{\frac{2}{r-1}}\right)\|\u-\v\|_{\widetilde{\L}^{r+1}},
\end{align}
for $r>3$, by using the interpolation inequality. For $r=3$, a calculation similar to \eqref{lip} yields 
\begin{align}
\|\B(\u)-\B(\v)\|_{\V'+\widetilde{\L}^{\frac{4}{3}}}&\leq \left(\|\u\|_{\widetilde{\L}^{4}}+\|\v\|_{\widetilde{\L}^{4}}\right)\|\u-\v\|_{\widetilde{\L}^{4}},
\end{align}
hence $\B(\cdot):\V\cap\widetilde{\L}^{4}\to\V'+\widetilde{\L}^{\frac{4}{3}}$ is a locally Lipschitz operator. 
\subsection{Nonlinear operator}\label{opeC}
Let us now consider the operator $\mathcal{C}_r(\u):=\mathcal{P}(|\u|^{r-1}\u)$. For convenience of notation, we use $\mathcal{C}$ for $\mathcal{C}_r$ in the rest of the paper.  It is immediate that $\langle\mathcal{C}(\u),\u\rangle =\|\u\|_{\widetilde{\L}^{r+1}}^{r+1}$ and the map $\mathcal{C}(\cdot):\widetilde{\L}^{r+1}\to\widetilde{\L}^{\frac{r+1}{r}}$ is Gateaux differentiable with its Gateaux derivative 
\begin{align}\label{Gaetu}
	\mathcal{C}'(\u)\v&=\left\{\begin{array}{cl}\mathcal{P}(\v),&\text{ for }r=1,\\ \left\{\begin{array}{cc}\mathcal{P}(|\u|^{r-1}\v)+(r-1)\mathcal{P}\left(\frac{\u}{|\u|^{3-r}}(\u\cdot\v)\right),&\text{ if }\u\neq \mathbf{0},\\\mathbf{0},&\text{ if }\u=\mathbf{0},\end{array}\right.&\text{ for } 1<r<3,\\ \mathcal{P}(|\u|^{r-1}\v)+(r-1)\mathcal{P}(\u|\u|^{r-3}(\u\cdot\v)), &\text{ for }r\geq 3,\end{array}\right.
\end{align}
for all $\v\in\widetilde{\L}^{r+1}$. Let us define the function of one variable $\varphi(\cdot):[0,1]\to\R^d$ by 
\begin{align*}
	\varphi(\theta):=|(1-\theta)\y+\theta\z|^{r-1}((1-\theta)\y+\theta\z), \text{ for  } \ \theta\in[0,1],
\end{align*}
where $\y,\z\in\R^d$. It is easy to see that
\begin{align*}
	\varphi(1)-\varphi(0)=-\left(|\y|^{r-1}\y-|\z|^{r-1}\z\right).
\end{align*}
One can easily compute the derivative of $\varphi$ as $\varphi'(\theta)=-r|(1-\theta)\y+\theta\z|^{r-1}(\y-\z)$. By an application of the Mean Value Theorem, we can estimate
\begin{align}\label{MVT}
	\left||\y|^{r-1}\y-|\z|^{r-1}\z\right|&=|\varphi(1)-\varphi(0)|\leq\max\limits_{\theta\in[0,1]}|\varphi'(\theta)|\nonumber\\&=\max\limits_{\theta\in[0,1]}\left|-r|(1-\theta)\y+\theta\z|^{r-1}(\y-\z)\right|\nonumber\\&\leq r\left(\max\limits_{\theta\in[0,1]}|(1-\theta)\y+\theta\z|^{r-1}\right)
	|\y-\z|\nonumber\\&\leq r\left(|\y|+|\z|\right)^{r-1}|\y-\z|.
\end{align}
By making the use of \eqref{MVT} and H\"older's inequality, we finally calculate
\begin{align}\label{213}
&\langle \mathcal{P}(|\u|^{r-1}\u)-\mathcal{P}(|\v|^{r-1}\v),\u-\v\rangle\nonumber\\
&=\langle (|\u|^{r-1}\u)-(|\v|^{r-1}\v),\u-\v\rangle\nonumber\\&\leq \||\u|^{r-1}\u-|\v|^{r-1}\v\|_{\widetilde{\L}^{\frac{r+1}{r}}}\|\u-\v\|_{\widetilde{\L}^{r+1}}\nonumber\\&\leq r\left(\|\u\|_{\widetilde{\L}^{r+1}}+\|\v\|_{\widetilde{\L}^{r+1}}\right)^{r-1}\|\u-\v\|_{\widetilde{\L}^{r+1}}^2,
\end{align}
for all $\u,\v\in\widetilde{\L}^{r+1}$.
Thus the operator $\mathcal{C}(\cdot):\widetilde{\L}^{r+1}\to\widetilde{\L}^{\frac{r+1}{r}}$ is locally Lipschitz. Moreover, for any $r\in[1,\infty)$, we have 

\begin{align}\label{2pp11}
&\langle \mathcal{P}(\u|\u|^{r-1})-\mathcal{P}(\v|\v|^{r-1}),\u-\v\rangle\nonumber\\&=\int_{\mathcal{O}}\left(\u(x)|\u(x)|^{r-1}-\v(x)|\v(x)|^{r-1}\right)\cdot(\u(x)-\v(x))\d x\nonumber\\&=\int_{\mathcal{O}}\left(|\u(x)|^{r+1}-|\u(x)|^{r-1}\u(x)\cdot\v(x)-|\v(x)|^{r-1}\u(x)\cdot\v(x)+|\v(x)|^{r+1}\right)\d x\nonumber\\&\geq\int_{\mathcal{O}}\left(|\u(x)|^{r+1}-|\u(x)|^{r}|\v(x)|-|\v(x)|^{r}|\u(x)|+|\v(x)|^{r+1}\right)\d x\nonumber\\&=\int_{\mathcal{O}}\left(|\u(x)|^r-|\v(x)|^r\right)(|\u(x)|-|\v(x)|)\d x\geq 0. 
\end{align}
Furthermore, we find 
\begin{align}\label{224}
&\langle\mathcal{P}(\u|\u|^{r-1})-\mathcal{P}(\v|\v|^{r-1}),\u-\v\rangle\nonumber\\&=\langle|\u|^{r-1},|\u-\v|^2\rangle+\langle|\v|^{r-1},|\u-\v|^2\rangle+\langle\v|\u|^{r-1}-\u|\v|^{r-1},\u-\v\rangle\nonumber\\&=\||\u|^{\frac{r-1}{2}}(\u-\v)\|_{\H}^2+\||\v|^{\frac{r-1}{2}}(\u-\v)\|_{\H}^2\nonumber\\&\quad+\langle\u\cdot\v,|\u|^{r-1}+|\v|^{r-1}\rangle-\langle|\u|^2,|\v|^{r-1}\rangle-\langle|\v|^2,|\u|^{r-1}\rangle.
\end{align}
But, we know that 
\begin{align*}
&\langle\u\cdot\v,|\u|^{r-1}+|\v|^{r-1}\rangle-\langle|\u|^2,|\v|^{r-1}\rangle-\langle|\v|^2,|\u|^{r-1}\rangle\nonumber\\&=-\frac{1}{2}\||\u|^{\frac{r-1}{2}}(\u-\v)\|_{\H}^2-\frac{1}{2}\||\v|^{\frac{r-1}{2}}(\u-\v)\|_{\H}^2+\frac{1}{2}\langle\left(|\u|^{r-1}-|\v|^{r-1}\right),\left(|\u|^2-|\v|^2\right)\rangle \nonumber\\&\geq -\frac{1}{2}\||\u|^{\frac{r-1}{2}}(\u-\v)\|_{\H}^2-\frac{1}{2}\||\v|^{\frac{r-1}{2}}(\u-\v)\|_{\H}^2.
\end{align*}
From \eqref{224}, we finally have 
\begin{align}\label{2.23}
&\langle\mathcal{P}(\u|\u|^{r-1})-\mathcal{P}(\v|\v|^{r-1}),\u-\v\rangle\geq \frac{1}{2}\||\u|^{\frac{r-1}{2}}(\u-\v)\|_{\H}^2+\frac{1}{2}\||\v|^{\frac{r-1}{2}}(\u-\v)\|_{\H}^2\geq 0,
\end{align}
for $r\geq 1$. 

%
%

\subsection{Monotonicity}
Let us now show the monotonicity as well as the hemicontinuity properties of the linear and nonlinear operators, which play a crucial role in this paper. Interested readers can refer \cite{VB1,VB,SHNP,HDJ} for further reading. 
\begin{definition}
	Let $\X$ be a Banach space and let $\X^{\prime}$ be its topological dual.
	\begin{enumerate} 
		\item [(a)] An operator $\G:\mathrm{D}\rightarrow
	\X^{\prime},$ where $\mathrm{D}=\mathrm{D}(\G)\subset \X$ is said to be
	\emph{monotone} if
	$$\langle\G(\x)-\G(\y),\x-\y\rangle\geq
	0,\ \text{ for all } \ \x,\y\in \mathrm{D}.$$  
	
\item [(b)] Let $\G:\D\to\X^{\prime}$ be a monotone operator. We say that $\G$ is \emph{maximal monotone operator} if
$\langle \w-\G(\y),\x-\y\rangle\geq 0$ on $\D$ implies that $\x\in\D$ and $\w=\G(\x)$. In particular, $\G$ is maximal monotone if it has no nontrivial monotone extension.

\item [(c)] The operator $\G(\cdot)$ is said to be \emph{hemicontinuous} at $\x\in\D$, if for every $\y\in\X$ and every $\lambda\in[0,\lambda_0)$, $\lambda_0>0$ with $\x+\lambda\y\in\D$, we have 
\begin{align*}
	\lim_{\lambda\to 0}\langle\G(\x+\lambda \y),\w\rangle=\langle\G(\x),\w\rangle
\end{align*}   
 for every $\w\in\X^{\prime}$. If this is true for every $\x\in\D$, then we say that $\G$ is hemicontinuous.

\item [(d)] The operator $\G(\cdot)$ is said to be \emph{demicontinuous}, at $\x\in\mathrm{D}$ if for every sequence $\{\x_n\}_{n\geq1}$ in $\D$ such that $\x_n\to \x$ in $\X$ implies $\G(\x_n)\xrightarrow{w}\G(\x)$ in $\X^{\prime}$. If this is true for every $\x\in\D$, then we say that $\G$ is demicontinuous.

\item [(e)] The operator $\G(\cdot)$ is called \emph{coercive} if $$\lim_{\|\x\|_{\X}\to\infty}\frac{\langle\G(\x),\x\rangle}{\|\x\|_{\X}}=+\infty.$$ 
	\end{enumerate} 
\end{definition}

\begin{remark}
	Clearly demicontinuity implies hemicontinuity. Moreover, if $\G(\cdot)$ is a monotone operator and $\D(\G)=\X$, then converse is also true (see \cite[Remark 1.24, Chapter 3]{SHNP}).
	\end{remark}

\begin{theorem}\label{thm2.2}
	Let $\u,\v\in\V\cap\widetilde{\L}^{r+1}$, for $r>3$. Then,	for the operator $\G(\u)=\mu \A\u+\B(\u)+\beta\mathcal{C}(\u)$, we  have 
	\begin{align}\label{fe}
	\langle(\G(\u)-\G(\v),\u-\v\rangle+\varrho\|\u-\v\|_{\H}^2\geq 0,
	\end{align}
	where
	 \begin{align}\label{215}
		\varrho=\frac{r-3}{2\mu(r-1)}\left(\frac{2}{\beta\mu (r-1)}\right)^{\frac{2}{r-3}}.
	\end{align} 
	That is, the operator $\G+\varrho\mathrm{I}_d$, where $\I_d$ is an identity operator on $\H$, is a monotone operator from $\V\cap\widetilde{\L}^{r+1}$ to $\V'+\widetilde{\L}^{\frac{r+1}{r}}$. 
\end{theorem}
\begin{proof}
	We estimate $	\langle\A\u-\A\v,\u-\v\rangle $ by	using an integration by parts as
	\begin{align}\label{ae}
	\langle\A\u-\A\v,\u-\v\rangle =\|\u-\v\|^2_{\V}.
	\end{align}
	From \eqref{2.23}, we easily have 
	\begin{align}\label{2.27}
	\beta	\langle\mathcal{C}(\u)-\mathcal{C}(\v),\u-\v\rangle \geq \frac{\beta}{2}\||\v|^{\frac{r-1}{2}}(\u-\v)\|_{\H}^2. 
	\end{align}
	Note that $\langle\B(\u,\u-\v),\u-\v\rangle=0$ and it implies that
	\begin{equation}\label{441}
	\begin{aligned}
	\langle \B(\u)-\B(\v),\u-\v\rangle &=\langle\B(\u,\u-\v),\u-\v\rangle +\langle \B(\u-\v,\v),\u-\v\rangle \nonumber\\&=\langle\B(\u-\v,\v),\u-\v\rangle=-\langle\B(\u-\v,\u-\v),\v\rangle.
	\end{aligned}
	\end{equation} 
	Using H\"older's and Young's inequalities, we estimate $|\langle\B(\u-\v,\u-\v),\v\rangle|$ as  
	\begin{align}\label{2p28}
	|\langle\B(\u-\v,\u-\v),\v\rangle|&\leq\|\u-\v\|_{\V}\|\v(\u-\v)\|_{\H}\nonumber\\&\leq\frac{\mu }{2}\|\u-\v\|_{\V}^2+\frac{1}{2\mu }\|\v(\u-\v)\|_{\H}^2.
	\end{align}
	We take the term $\|\v(\u-\v)\|_{\H}^2$ from \eqref{2p28} and use H\"older's and Young's inequalities to estimate it as (see \cite{KWH} also)
	\begin{align}\label{2.29}
	&\int_{\mathcal{O}}|\v(x)|^2|\u(x)-\v(x)|^2\d x\nonumber\\&=\int_{\mathcal{O}}|\v(x)|^2|\u(x)-\v(x)|^{\frac{4}{r-1}}|\u(x)-\v(x)|^{\frac{2(r-3)}{r-1}}\d x\nonumber\\&\leq\left(\int_{\mathcal{O}}|\v(x)|^{r-1}|\u(x)-\v(x)|^2\d x\right)^{\frac{2}{r-1}}\left(\int_{\mathcal{O}}|\u(x)-\v(x)|^2\d x\right)^{\frac{r-3}{r-1}}\nonumber\\&\leq{\beta\mu }\left(\int_{\mathcal{O}}|\v(x)|^{r-1}|\u(x)-\v(x)|^2\d x\right)\nonumber\\&\quad+\frac{r-3}{r-1}\left(\frac{2}{\beta\mu (r-1)}\right)^{\frac{2}{r-3}}\left(\int_{\mathcal{O}}|\u(x)-\v(x)|^2\d x\right),
	\end{align}
	for $r>3$. Using \eqref{2.29} in \eqref{2p28}, we find 
	\begin{align}\label{2.30}
	&|\langle\B(\u-\v,\u-\v),\v\rangle|\nonumber\\&\leq\frac{\mu }{2}\|\u-\v\|_{\V}^2+\frac{\beta}{2}\||\v|^{\frac{r-1}{2}}(\u-\v)\|_{\H}^2+\varrho\|\u-\v\|_{\H}^2.
	\end{align}
	Combining \eqref{ae}, \eqref{2.27} and \eqref{2.30}, we get 
	\begin{align}
	\langle(\G(\u)-\G(\v),\u-\v\rangle+\varrho\|\u-\v\|_{\H}^2\geq\frac{\mu }{2}\|\u-\v\|_{\V}^2\geq 0,
	\end{align}
	for $r>3$ and the estimate \eqref{fe} follows.  
\end{proof} 
\begin{theorem}\label{thm2.3}
	For the critical case $r=3$ with $2\beta\mu \geq 1$, the operator $\G(\cdot):\V\cap\widetilde{\L}^{4}\to \V'+\widetilde{\L}^{\frac{4}{3}}$ is globally monotone, that is, for all $\u,\v\in\V$, we have 
	\begin{align}\label{218}\langle\G(\u)-\G(\v),\u-\v\rangle\geq 0.\end{align}
\end{theorem}
\begin{proof}
	From \eqref{2.23}, we have 
	\begin{align}\label{231}
	\beta\langle\mathcal{C}(\u)-\mathcal{C}(\v),\u-\v\rangle\geq\frac{\beta}{2}\|\v(\u-\v)\|_{\H}^2. 
	\end{align}
	We estimate $|\langle\B(\u-\v,\u-\v),\v\rangle|$ using H\"older's and Young's inequalities as 
	\begin{align}\label{232}
	|\langle\B(\u-\v,\u-\v),\v\rangle|\leq\|\v(\u-\v)\|_{\H}\|\u-\v\|_{\V} \leq\mu \|\u-\v\|_{\V}^2+\frac{1}{4\mu }\|\v(\u-\v)\|_{\H}^2.
	\end{align}
	Combining \eqref{ae}, \eqref{231} and \eqref{232}, we obtain 
	\begin{align}
	\langle\G(\u)-\G(\v),\u-\v\rangle\geq\frac{1}{2}\left(\beta-\frac{1}{2\mu }\right)\|\v(\u-\v)\|_{\H}^2\geq 0,
	\end{align}
	provided $2\beta\mu \geq 1$. 
\end{proof}
\begin{remark}\label{rem2.4}
	1. As in \cite{YZ}, for $r\geq 3$, one can estimate $|\langle\B(\u-\v,\u-\v),\v\rangle|$ as 
	\begin{align}\label{2.26}
	&	|\langle\B(\u-\v,\u-\v),\v\rangle|\nonumber\\&\leq \mu \|\u-\v\|_{\V}^2+\frac{1}{4\mu }\int_{\mathcal{O}}|\v(x)|^2|\u(x)-\v(x)|^2\d x\nonumber\\&= \mu \|\u-\v\|_{\V}^2+\frac{1}{4\mu }\int_{\mathcal{O}}|\u(x)-\v(x)|^2\left(|\v(x)|^{r-1}+1\right)\frac{|\v(x)|^2}{|\v(x)|^{r-1}+1}\d x\nonumber\\&\leq \mu \|\u-\v\|_{\V}^2+\frac{1}{4\mu }\int_{\mathcal{O}}|\v(x)|^{r-1}|\u(x)-\v(x)|^2\d x+\frac{1}{4\mu }\int_{\mathcal{O}}|\u(x)-\v(x)|^2\d x,
	\end{align}
	where we used the fact that $\left\|\frac{|\v|^2}{|\v|^{r-1}+1}\right\|_{\widetilde{\L}^{\infty}}<1$, for $r\geq 3$. The above estimate yields a local monotonicity result given in \eqref{218}, provided $2\beta\mu \geq 1$. 
	
	2. For $d=2$ and $r= 3$, one can estimate $|\langle\B(\u-\v,\u-\v),\v\rangle|$ using H\"older's, Ladyzhenskaya and Young's inequalities  as 
	\begin{align}\label{2.21}
	|\langle\B(\u-\v,\u-\v),\v\rangle|&\leq \|\u-\v\|_{\widetilde\L^4}\|\u-\v\|_{\V}\|\v\|_{\widetilde\L^4}\nonumber\\&\leq 2^{1/4}\|\u-\v\|_{\H}^{1/2}\|\u-\v\|_{\V}^{3/2}\|\v\|_{\widetilde\L^4}\nonumber\\&\leq  \frac{\mu }{2}\|\u-\v\|_{\V}^2+\frac{27}{16\mu^3} \|\v\|_{\widetilde{\L}^4}^4\|\u-\v\|_{\H}^2.
	\end{align}
	Combining \eqref{ae}, \eqref{2.27} and \eqref{2.21}, we obtain 
	\begin{align}\label{fe2}
	\langle\G(\u)-\G(\v),\u-\v\rangle+ \frac{27}{32\mu ^3}N^4\|\u-\v\|_{\H}^2\geq 0,
	\end{align}
	for all $\v\in\widehat{\mathbb{B}}_N$, where $\widehat{\mathbb{B}}_N$ is an $\widetilde{\L}^4$-ball of radius $N$, that is,
	$
	\widehat{\mathbb{B}}_N:=\big\{\z\in\widetilde{\L}^4:\|\z\|_{\widetilde{\L}^4}\leq N\big\}.
	$ Thus, the operator $\G(\cdot)$ is locally monotone in this case (see \cite{MJSS,MTM}, etc). 
\end{remark}

Let us now show that the operator $\G:\V\cap\widetilde{\L}^{r+1}\to \V'+\widetilde{\L}^{\frac{r+1}{r}}$ is hemicontinuous, which is useful in proving the existence of weak solutions and strong solutions to the system \eqref{1}. 
\begin{lemma}\label{lem2.8}
	The operator $\G:\V\cap\widetilde{\L}^{r+1}\to \V'+\widetilde{\L}^{\frac{r+1}{r}}$ is demicontinuous. 
\end{lemma}
\begin{proof}
	Let us take a sequence $\u^n\to \u$ in $\V\cap\widetilde{\L}^{r+1}$, that is, $\|\u^n-\u\|_{\widetilde\L^{r+1}}+\|\u^n-\u\|_{\V}\to 0$, as $n\to\infty$. For any $\v\in\V\cap\widetilde{\L}^{r+1}$, we consider 
	\begin{align}\label{214}
	\langle\G(\u^n)-\G(\u),\v\rangle&=\mu \langle \A(\u^n)-\A(\u),\v\rangle+\langle\B(\u^n)-\B(\u),\v\rangle+\beta\langle \mathcal{C}(\u^n)-\mathcal{C}(\u),\v\rangle.
	\end{align} 
	Taking $\langle \A(\u^n)-\A(\u),\v\rangle$ from \eqref{214}, we estimate it as 
	\begin{align}
	|\langle \A(\u^n)-\A(\u),\v\rangle|=|(\nabla(\u^n-\u),\nabla\v)|\leq\|\u^n-\u\|_{\V}\|\v\|_{\V}\to 0, \ \text{ as } \ n\to\infty, 
	\end{align}
	since $\u^n\to \u$ in $\V$. We estimate the term $\langle\B(\u^n)-\B(\u),\v\rangle$ from \eqref{214} using H\"older's inequality as 
	\begin{align}
	|\langle\B(\u^n)-\B(\u),\v\rangle|&=|\langle\B(\u^n,\u^n-\u),\v\rangle+\langle\B(\u^n-\u,\u),\v\rangle|\nonumber\\&
	\leq|\langle\B(\u^n,\v),\u^n-\u\rangle|+|\langle\B(\u^n-\u,\v),\u\rangle|\nonumber\\&\leq\left(\|\u^n\|_{\widetilde{\L}^{\frac{2(r+1)}{r-1}}}+\|\u\|_{\widetilde{\L}^{\frac{2(r+1)}{r-1}}}\right)\|\u^n-\u\|_{\widetilde{\L}^{r+1}}\|\v\|_{\V}\nonumber\\&\leq \left(\|\u^n\|_{\H}^{\frac{r-3}{r-1}}\|\u^n\|_{\widetilde{\L}^{r+1}}^{\frac{2}{r-1}}+\|\u\|_{\H}^{\frac{r-3}{r-1}}\|\u\|_{\widetilde{\L}^{r+1}}^{\frac{2}{r-1}}\right)\|\u^n-\u\|_{\widetilde{\L}^{r+1}}\|\v\|_{\V}\nonumber\\& \to 0, \ \text{ as } \ n\to\infty, 
	\end{align}
	since $\u^n\to\u$ in $\widetilde\L^{r+1}$ and $\u^n,\u\in\V\cap\widetilde\L^{r+1}$. The final term $\langle \mathcal{C}(\u^n)-\mathcal{C}(\u),\v\rangle$ in \eqref{214} can be estimated in a similar way as in \eqref{213} and we conclude that
	\begin{align}
	|\langle \mathcal{C}(\u^n)-\mathcal{C}(\u),\v\rangle|\leq r\|\u^n-\u\|_{\widetilde{\L}^{r+1}}\left(\|\u^n\|_{\widetilde{\L}^{r+1}}+\|\u\|_{\widetilde{\L}^{r+1}}\right)^{r-1}\|\v\|_{\widetilde{\L}^{r+1}}\to 0, 
	\end{align}
 as $n\to\infty$, since $\u^n\to\u$ in $\widetilde{\L}^{r+1}$ and $\u^n,\u\in\V\cap\widetilde{\L}^{r+1}$, for $r\geq 3$. From the above convergences, it is immediate that $\langle\G(\u^n)-\G(\u),\v\rangle \to 0$, for all $\v\in \V\cap\widetilde{\L}^{r+1}$.
	Hence the operator $\G:\V\cap\widetilde{\L}^{r+1}\to \V'+\widetilde{\L}^{\frac{r+1}{r}}$ is demicontinuous, which implies that the operator $\G(\cdot)$ is hemicontinuous also. 
\end{proof}

\section{Existence and uniqueness of weak solutions}\label{sec3}\setcounter{equation}{0}
In this section, we provide an abstract formulation of the system \eqref{1} and prove the existence and uniqueness of weak solutions using the monotonicity property of the linear and nonlinear operators and the Minty-Browder technique. 
\subsection{Abstract formulation and weak solution} We take the Helmholtz-Hodge orthogonal projection $\mathcal{P}$ in \eqref{1} to obtain  the abstract formulation for  $t\in(0,T)$ as: 
\begin{equation}\label{kvf}
\left\{
\begin{aligned}
\frac{\d\u(t)}{\d t}+\mu \A\u(t)+\B(\u(t))+ \beta\mathcal{C}(\u(t))&=\f(t),\\
\u(0)&=\u_0\in\H,
\end{aligned}
\right.
\end{equation}
for $r\geq 3$, where $\f\in\mathrm{L}^2(0,T;\V')$. Strictly speaking one should use $\mathcal{P}\f$ instead of $\f$, for simplicity, we use $\f$. Let us now provide the definition of \emph{weak solution} of the system \eqref{kvf} for $r\geq 3$.
\begin{definition}\label{weakd}
For $r\geq 3$,	a function  $$\u\in\mathrm{L}^{\infty}(0,T;\H)\cap\mathrm{L}^2(0,T;\V)\cap\mathrm{L}^{r+1}(0,T;\widetilde\L^{r+1})),$$  with $\frac{\d\u}{\d t}\in\mathrm{L}^{2}(0,T;\mathbb{V}')+\mathrm{L}^{\frac{r+1}{r}}(0,T;\widetilde\L^{\frac{r+1}{r}}),$  is called a \emph{weak solution} to the system (\ref{kvf}), if for $\f\in\mathrm{L}^2(0,T;\V')$ and  $\u_0\in\H$, $\u(\cdot)$ satisfies:
	\begin{align}\label{3.13}
&-\int_{t_0}^{t_1}\langle\u(s),\partial_t\v(s)\rangle\d s+\mu\int_{t_0}^{t_1}(\nabla\u(s),\nabla\v(s))\d s+\int_{t_0}^{t_1}\langle(\u(s)\cdot\nabla)\u(s),\v(s)\rangle\d s\nonumber\\&+\beta\int_{t_0}^{t_1}\langle\u(s)|\u(s)|^{r-1},\v(s)\rangle\d s= -\langle\u(t_1),\v(t_1)\rangle+\langle\u(t_0),\v(t_0)\rangle+\int_{t_0}^{t_1}\langle\f(s),\v(s)\rangle\d s,
	\end{align}
for all test functions $\v\in\mathcal{V}_T$, almost all initial times $t_0\in[0,T)$, including zero, and almost every $t_1\in(t_0,T)$. A function $\u$ is called a \emph{global weak solution} if it is a weak solution for all $T>0$.
\end{definition}
\begin{definition}
	A \emph{Leray-Hopf weak solution} of the convective Brinkman--Forchheimer equations \eqref{kvf} with the initial condition $\u_0\in\H$ is a weak solution satisfying the following \emph{strong energy inequality}:
	\begin{align*}
	\|\u(t_1)\|_{\H}^2+\mu\int_{t_0}^{t_1}\|\u(s)\|_{\V}^2\d s+2\beta\int_{t_0}^{t_1}\|\u(s)\|_{\widetilde\L^{r+1}}^{r+1}\d s\leq\|\u(t_0)\|_{\H}^2+\frac{1}{\mu}\int_{t_0}^{t_1}\|\f(s)\|_{\V'}^2\d s,
	\end{align*}
	for almost every $t_0\in[0, T)$, including zero, and all $t_1\in(t_0, T)$. 
\end{definition}

\begin{remark}
The regularity  $\u\in\mathrm{L}^{\infty}(0,T;\H)\cap\mathrm{L}^2(0,T;\V)\cap\mathrm{L}^{r+1}(0,T;\widetilde\L^{r+1}))$ and  $\frac{\d\u}{\d t}\in\mathrm{L}^{2}(0,T;\mathbb{V}')+\mathrm{L}^{\frac{r+1}{r}}(0,T;\widetilde\L^{\frac{r+1}{r}})\hookrightarrow \mathrm{L}^{\frac{r+1}{r}}(0,T;\mathbb{V}'+\widetilde\L^{\frac{r+1}{r}}) $ imply $\u\in\W^{1,\frac{r+1}{r}}(0,T;\mathbb{V}'+\widetilde\L^{\frac{r+1}{r}})\hookrightarrow\C([0,T];\mathbb{V}'+\widetilde\L^{\frac{r+1}{r}}) $ (\cite[Theorem 2, pp. 302]{LCE}). Since $\H$ is reflexive and the embedding $\H\hookrightarrow\mathbb{V}'+\widetilde\L^{\frac{r+1}{r}}$ is continuous,   therefore, by an application of \cite[Lemma 8.1, Chapter 3, pp. 275]{JLLEM} (also see \cite[Proposition 1.7.1, Chapter 1, pp. 61]{PCAM}) yields $\u\in\C_w([0,T];\H)$, where $\u\in\C_w([0,T];\H)$ denotes the space of functions $\u:[0,T]\to \H$ which are weakly continuous. That is, for all $\boldsymbol{\zeta}\in\H$, the scalar function $$[0,T]\ni t\mapsto (\u(t),\boldsymbol{\zeta})\in\R$$ is continuous on $[0,T]$. Therefore, the first two terms appearing in the right hand side of \eqref{3.13} make sense. 
\end{remark}

From \cite{SNA}, it is known that for every $\u_0\in\H$, there exists at least one global Leray-Hopf weak solution of the system \eqref{kvf}. For $r\geq 3$, it has been proved in \cite{CLF,KWH,KT2} that every weak solution of \eqref{kvf} with the initial condition $\u_0\in\H$ satisfies the energy equality: 
	\begin{align}\label{33} 
\|\u(t_1)\|_{\H}^2+2\mu\int_{t_0}^{t_1}\|\u(s)\|_{\V}^2\d s+2\beta\int_{t_0}^{t_1}\|\u(s)\|_{\widetilde\L^{r+1}}^{r+1}\d s=\|\u(t_0)\|_{\H}^2+2\int_{t_0}^{t_1}\langle\f(s),\v(s)\rangle\d s.
\end{align}
Our main goal is to prove the existence and uniqueness of Leray-Hopf weak solution to the system \eqref{kvf} satisfying the energy equality \eqref{33}. As we are using the eigenspace approximation introduced in \cite{CLF}, we replace $\mathcal{V}_T$ with the space ${\mathcal{V}}_{\infty}$ consisting of finite combinations of eigenfunctions of the Stokes operator (see \cite{JCR4}). We define 
\begin{align}\label{eqn-vinfty}
{\mathcal{V}}_{\infty}:=\left\{\v: \v=\sum_{k=1}^N\alpha_k(t)\boldsymbol{w}_k(x):\alpha_k\in\C_0^1([0,\infty)), \ \text{ for some } \ N\in\mathbb{N}\right\},
\end{align}
where $\{\boldsymbol{w}_k\}_{k=1}^{\infty}$ are the eigenfunctions of the Stokes operator (see \cite[Theorem 2.24, Chapter 2, pp. 60]{JCR4}).

\subsection{Energy estimates}\label{sec2.3}
We use the Faedo-Galerkin approximation  method to show the existence of a weak solution of \eqref{kvf}. Note that since $\V$ is separable Hilbert space and $\V$ is dense in $\H$, we can find a linearly independent sequence $\{\boldsymbol{w}_j\}_{j\in\N}$ which is total in $\V$. In particular, one can take $\{\boldsymbol{w}_j\}_{j\in\N}$ as the eigenfunctions of the Stokes operator. By using an Gram-Schmidt orthogonalization process, we may assume that $\{\boldsymbol{w}_j\}_{j\in\N}$ is an orthonormal basis in $\H$. Let us denote by $\H_n:=\mathrm{span}\{\boldsymbol{w}_j:j=1,\ldots,n, \ n\in\N\}$, with the norm inherited from $\H$, and $\V_n:=\mathrm{span}\{\boldsymbol{w}_j:j=1,\ldots,n, \ n\in\N\}$, with the norm inherited from $\V$. We also denote by $\P_n$, the orthogonal projection of $\V'$ to $\H_n$, that is, for every $\x\in\V'$, we write $\P_n\x=\sum\limits_{j=1}^n \langle\x,\boldsymbol{w}_j\rangle \boldsymbol{w}_j$. Since, every element $x\in\H$, induces a functional $\x^*\in\V'$ by the formula $\langle\x^*,\y\rangle=\langle\x,\y\rangle$, $\y\in\V$, then $\P_n|_\H$, is the orthogonal projection from $\H$ onto $\H_n$. Hence in particular, $\mathrm{P}_n$ is the orthogonal projection from $\H$ onto $\H_n$, given by $\mathrm{P}_n\x=\sum\limits_{i=1}^n\left(\x,\boldsymbol{w}_i\right)\boldsymbol{w}_i$, so that $\|\mathrm{P}_n\x-\x\|_{\H}\to 0$ as $n\to\infty$ for all $\x\in\H$. 
Let $\mathfrak{j}_n:\H_n\hookrightarrow\H$ be the natural embedding and let us define $\A _n:=\P_n\A\mathfrak{j}_n$, $\B_n:=\P_n\B(\mathfrak{j}_n)$,  $\mathcal{C}_n:=\P_n\mathcal{C}(\mathfrak{j}_n)$ and $\f_n:=\P_n\f$.
Let us now fix $T>0$. For each $n\in\N$, we search for a approximate solution of the form 
\begin{align*}
	\u^n(x,t):=\sum\limits_{k=1}^n g_k^n(t)\boldsymbol{w}_k(x),
\end{align*}
where $g_1^n(t),\ldots,g_k^n(t)$ are unknown scalar functions of $t$ such that it satisfies the following finite-dimensional system of ordinary differential equations in $\H_n$:
\begin{equation}\label{appxode1}
\left\{
\begin{aligned}
\frac{\d}{\d t}\left(\u^n(t),\boldsymbol{w}_j\right)&=-(\mu \mathrm{A}_n\u^n(t)+\mathrm{B}_n(\u^n(t))+\mathcal{C}_n(\u^n(t))-\f^n(t),\boldsymbol{w}_j),\\
(\u^n(0),\boldsymbol{w}_j)&=(\u_0^n,\boldsymbol{w}_j),
\end{aligned}
\right.
\end{equation}
for a.e. $t\in[0,T]$ and $j=1,\ldots,n$, where $\u_0^n=\mathrm{P}_n\u_0$.
Since $\B_n(\cdot)$ and $\mathcal{C}_n(\cdot)$ are locally Lipschitz (see \eqref{lip} and \eqref{213}), therefore by using the Carath\'eodory's existence theorem, there exists a local maximal solution $\u^n\in\C([0,T^*];\H_n),$  for some $0<T^*\leq T$ of the system  \eqref{appxode1} and uniqueness is immediate from the local Lipschitz property. The time $T^*$ can be extended to $T$ by establishing the uniform energy estimates of the solutions satisfied by the system \eqref{appxode1}. Note that the energy estimate is true for any $r\in[1,\infty)$.

\begin{proposition}[A-priori estimates]\label{prop4.5}
	Let $\u_0\in \H$ and $\f\in\mathrm{L}^2(0,T;\V')$  be given. Then, for $r\in[1,\infty)$, we have 
	\begin{align}\label{energy1}
	\sup_{t\in[0,T]}\|\u^n(t)\|_{\H}^2+\mu \int_0^T\|\u^n(t)\|_{\V}^2\d t+2\beta\int_0^T\|\u^n(t)\|_{\widetilde\L^{r+1}}^{r+1}\d t\leq \|\u_0\|_{\H}^2+\frac{1}{\mu }\int_0^T\|\f(t)\|_{\V'}^2\d t.
	\end{align}
\end{proposition} 
\begin{proof}
 From the definition of the operators $\A$, $\B(\cdot)$ and $\mathcal{C}(\cdot)$ (see Subsections \ref{opeA}-\ref{opeC}), it is observed that
	\begin{align*}
		(\A_n\u^n,\u^n)&=(\P_n\A\u^n,\u^n)=(\A\u^n,\u^n)=\|\u^n\|_{\V}^2,\\
		(\B_n(\u^n),\u^n)&=(\P_n\B(\u^n),\u^n)=(\B(\u^n),\u^n)=0,\\
		(\mathcal{C}_n(\u^n),\u^n)&=(\P_n\mathcal{C}(\u^n),\u^n)=(\mathcal{C}(\u^n),\u^n)=\|\u^n\|_{\widetilde\L^{r+1}}^{r+1}.
	\end{align*}
	Now, multiplying \eqref{appxode1} by $g_k^n(\cdot)$, summing up over $k=1,\ldots,n$ and using the above identities, we obtain
	\begin{align*}
	&\frac{1}{2}\frac{\d}{\d t}\|\u^n(t)\|_{\H}^2+\mu \|\u^n(t)\|_{\V}^2+\beta\|\u^n(t)\|_{\widetilde\L^{r+1}}^{r+1}=(\f^n(t),\u^n(t)),
	\end{align*}
	 Integrating the equality from $0$ to $t$, we find 
	\begin{align}\label{415}
	&\|\u^n(t)\|_{\H}^2+2\mu \int_0^t\|\u^n(s)\|_{\V}^2\d s+2\beta\int_0^t\|\u^n(s)\|_{\widetilde\L^{r+1}}^{r+1}\d s\nonumber\\&= \|\u^n_0\|_{\H}^2+2 \int_0^t(\f^n(s),\u^n(s))\d s.
	\end{align}
	Using Cauchy-Schwarz and Young's inequalities, we estimate $|(\f^n,\u^n)|$ as 
	\begin{align*}
	|(\f^n,\u^n)|\leq \|\f^n\|_{\V'}\|\u^n\|_{\V}\leq \frac{1}{2\mu }\|\f^n\|_{\V'}^2+\frac{\mu }{2}\|\u^n\|_{\V}^2.
	\end{align*}
	Thus, using the fact that $\|\u_0^n\|_{\H}=\|\mathrm{P}_n\u_0\|_{\H}\leq \|\u_0\|_{\H}$ and $\|\f^n\|_{\V'}\leq \|\f\|_{\V'}$ in \eqref{415}, we infer 
	\begin{align}\label{417}
	&\|\u^n(t)\|_{\H}^2+\mu \int_0^t\|\u^n(s)\|_{\V}^2\d s+2\beta\int_0^t\|\u^n(s)\|_{\widetilde\L^{r+1}}^{r+1}\d s\leq \|\u_0\|_{\H}^2+\frac{1}{\mu }\int_0^t\|\f(s)\|_{\V'}^2\d s,
	\end{align}
for all $t\in[0,T]$	 and \eqref{energy1} follows. 
\end{proof}

\subsection{Global existence and uniqueness} Let us now show that the system \eqref{kvf} has a unique weak solution by making use of the Faedo-Galerkin approximation technique and energy estimates obtained in the previous subsection. The monotonicity property of the linear and nonlinear operators obtained in Theorem \ref{thm2.2}, and  the Minty-Browder technique are used to obtain global solvability results. 
\begin{theorem}\label{main}
For $2\leq d\leq 4$, let $\u_0\in \H$ and $\f\in\mathrm{L}^{2}(0,T;\V')$  be given.  Then  there exists a unique weak solution to the system (\ref{kvf}) satisfying $$\u\in\mathrm{C}([0,T];\H)\cap\mathrm{L}^2(0,T;\V)\cap\mathrm{L}^{r+1}(0,T;\widetilde\L^{r+1}),$$ for $r>3$.
\end{theorem}
\begin{proof}
		\noindent\textbf{Part (1). Existence:}  We prove the existence of a weak solution to the system \eqref{kvf}  in the following steps:
	\vskip 0.1in 
	\noindent\textbf{Step (i):} \emph{Finite-dimensional (Galerkin) approximation of the system (\ref{kvf}):}
	We consider the following system of ODEs in $\H_n$: 
	\begin{equation}\label{420}
	\left\{
	\begin{aligned}
	\frac{\d}{\d t}(\u^n(t),\v)+(\mu \mathrm{A}\u^n(t)+\mathrm{B}_n(\u^n(t))+\beta\mathcal{C}_n(\u^n(t)),\v)&=(\f^n(t),\v),\\
	(\u^n(0),\v)&=(\u_0^n,\v),
	\end{aligned}
	\right. 
	\end{equation}
	for any $\v \in\H_n$. Let us define the operator 
		\begin{align}\label{eqn-gn}
			\G_n(\cdot):=\mu\A_n+\B_n(\cdot)+\beta\mathcal{C}_n(\cdot).
		\end{align}
	 Then the system \eqref{420} can be rewritten as 
	\begin{equation}\label{421}
	\left\{
	\begin{aligned}
	\frac{\d}{\d t}(\u^n(t),\v)+(\G_n(\u^n(t)),\v)&=(\f^n(t),\v), \\
	(\u^n(0),\v)&=(\u_0^n,\v),
	\end{aligned}
	\right.
	\end{equation} 
	for any $\v\in\H_n$. Let us take the test function as $\u^n(\cdot)$ in \eqref{421} to derive the following energy equality: 
	\begin{align}\label{422}
	&	\|\u^n(t)\|_{\H}^2+2\int_0^t(\G_n(\u^n(s))-\f^n(s),\u^n(s))\d s=\|\u^n(0)\|_{\H}^2,
	\end{align}
	for any $t\in[0,T]$. Note that for $\varrho=\frac{r-3}{r-1}\left(\frac{2}{\beta\mu (r-1)}\right)^{\frac{2}{r-3}}$, with $r>3$, one can easily obtain the following identity:
		\begin{align}\label{eLd}
			\frac{\d}{\d t}\left(e^{-\varrho t}\u^n(t)\right)=
			e^{-\varrho t}\frac{\d\u^n(t)}{\d t}-\varrho e^{-\varrho t}\u^n(t),
		\end{align}
for a.e. $t\in[0,T]$.	By using \eqref{eLd}, we calculate
	\begin{align*}
		\frac{1}{2}\frac{\d}{\d t}\left(e^{-2\varrho t}\|\u^n(t)\|_{\H}^2\right)&= \bigg(\frac{\d}{\d t}\left(e^{-\varrho t}\u^n(t)\right),e^{-\varrho t}\u^n(t)\bigg) \nonumber\\&=
		e^{-2\varrho t}\bigg(\frac{\d\u^n(t)}{\d t},\u^n(t)\bigg)-\varrho e^{-2\varrho t} \|\u^n(t)\|_{\H}^2.
	\end{align*}
On integrating the above equality from $0$ to $t$ and using \eqref{421}, it can be easily shown that 
	\begin{align}\label{eqn42}
	&e^{-2\varrho t}\|\u^n(t)\|^2_{\H} +2\int_0^te^{-2\varrho s}(\G_n(\u^n(s))-\f^n(s)+\varrho\u^n(s),\u^n(s))\d
	s=\|\u^n(0)\|_{\H}^2,
	\end{align}
	for any $t\in[0,T]$.

	\vskip 0.1in
	\noindent\textbf{Step (ii):} \emph{Weak convergence of the sequences $\u^n(\cdot)$ and $\G_n(\u^n(\cdot))$:} We know that the dual of $\mathrm{L}^1(0,T;\H)$ is $\mathrm{L}^{\infty}(0,T;\H)$, that is,  $\left(\mathrm{L}^1(0,T;\H)\right)' \cong \mathrm{L}^{\infty}(0,T;\H)$ and the space $\mathrm{L}^1(0,T;\H)$ is separable. Moreover, the spaces $\mathrm{L}^2(0,T;\V)$ and $\mathrm{L}^{r+1}(0,T;\widetilde\L^{r+1})$ are reflexive.   Making use of the  energy estimate in  Proposition \ref{prop4.5} and the Banach-Alaoglu theorem (or Helly's theorem), we can extract subsequences $\{\u^{n_k}(\cdot)\}$ and $\{\G_n(\u^{n_k}(\cdot))\}$ such that  (for notational convenience, we use $\{\u^{n}(\cdot)\}$ and $\{\G_n(\u^{n}(\cdot))\}$)
	\begin{equation}\label{3.20}
	\left\{
	\begin{aligned}
	\u^n&\xrightarrow{w^*}\u\ \text{ in }\ \mathrm{L}^{\infty}(0,T;\H),\\
	\u^n&\xrightarrow{w}\u\ \text{ in }\ \mathrm{L}^{2}(0,T;\V),\\
		\u^n&\xrightarrow{w}\u\ \text{ in }\ \mathrm{L}^{r+1}(0,T;\widetilde\L^{r+1})),\\ 
	\G_n(\u^n)&\xrightarrow{w}\G_0\ \text{ in }\ \mathrm{L}^{2}(0,T;\V')+\mathrm{L}^{\frac{r+1}{r}}(0,T;\widetilde\L^{\frac{r+1}{r}}).
	\end{aligned}
	\right.\end{equation}
Note that for $r\geq 3$,	the final convergence in (\ref{3.20}) can be justified using H\"older's inequality, interpolation inequality \eqref{211} and \eqref{energy1} as follows: 
{\small{	\begin{align}\label{428}
&\left|\int_0^T\langle\G_n(\u^n(t)),\boldsymbol{\varphi}(t)\rangle\d t\right|\nonumber\\&\leq\mu \left|\int_0^T(\nabla\u^n(t),\nabla\boldsymbol{\varphi}(t))\d t\right|+\left|\int_0^T\langle \B(\u^n(t),\boldsymbol{\varphi}(t)),\u^n(t)\rangle\d t\right|\nonumber\\&\quad+\beta\left|\int_0^T\langle|\u^n(t)|^{r-1}\u^n(t),\boldsymbol{\varphi}(t)\rangle\d t\right|\nonumber\\&\leq\mu \int_0^T\|\nabla\u^n(t)\|_{\H}\|\nabla\boldsymbol{\varphi}(t)\|_{\H}\d t+\int_0^T\|\u^n(t)\|_{\widetilde\L^{r+1}}\|\u^n(t)\|_{\widetilde\L^{\frac{2(r+1)}{r-1}}}\|\boldsymbol{\varphi}(t)\|_{\V}\d t\nonumber\\&\quad+\beta\int_0^T\|\u^n(t)\|_{\widetilde\L^{r+1}}^{r}\|\boldsymbol{\varphi}(t)\|_{\widetilde\L^{r+1}}\d t\nonumber\\&\leq \left[\mu \left(\int_0^T\|\nabla\u^n(t)\|_{\H}^2\d t\right)^{1/2}+\left(\int_0^T\|\u^n(t)\|_{\widetilde\L^{r+1}}^{\frac{2(r+1)}{r-1}}\|\u^n(t)\|_{\H}^{\frac{2(r-3)}{r-1}}\d t\right)^{1/2}\right]\left(\int_0^T\|\boldsymbol{\varphi}(t)\|_{\V}^2\d t\right)^{1/2}\nonumber\\&\quad+\beta\left(\int_0^T\|\u^n(t)\|_{\widetilde\L^{r+1}}^{r+1}\d t\right)^{\frac{r}{r+1}}\left(\int_0^T\|\boldsymbol{\varphi}(t)\|_{\widetilde\L^{r+1}}^{r+1}\d t\right)^{\frac{1}{r+1}}\nonumber\\&\leq \left[\mu \left(\int_0^T\|\u^n(t)\|_{\V}^2\d t\right)^{1/2}+\left(\int_0^T\|\u^n(t)\|_{\widetilde\L^{r+1}}^{r+1}\d t\right)^{\frac{1}{r-1}}\left(\int_0^T\|\u^n(t)\|_{\H}^2\d t\right)^{\frac{r-3}{2(r-1)}}\right]\nonumber\\&\quad\times\left(\int_0^T\|\boldsymbol{\varphi}(t)\|_{\V}^2\d t\right)^{1/2}+\beta\left(\int_0^T\|\u^n(t)\|_{\widetilde\L^{r+1}}^{r+1}\d t\right)^{\frac{r}{r+1}}\left(\int_0^T\|\boldsymbol{\varphi}(t)\|_{\widetilde\L^{r+1}}^{r+1}\d t\right)^{\frac{1}{r+1}}\nonumber\\&\leq C(\|\u_0\|_{\H},\mu ,T,\beta,\|\f\|_{\mathrm{L}^2(0,T;\V')})\left[\left(\int_0^T\|\boldsymbol{\varphi}(t)\|_{\V}^2\d t\right)^{1/2}+\left(\int_0^T\|\boldsymbol{\varphi}(t)\|_{\widetilde\L^{r+1}}^{r+1}\d t\right)^{\frac{1}{r+1}}\right],
	\end{align}}}
for all $\boldsymbol{\varphi}\in\mathrm{L}^2(0,T;\V)\cap\mathrm{L}^{r+1}(0,T;\widetilde\L^{r+1}))$. Note also that 
	\begin{align*}
	&\left|\int_0^T\left\langle\frac{\d\u^n}{\d t},\boldsymbol{\varphi}(t)\right\rangle\d t\right|\nonumber\\&\leq \left|\int_0^T\langle\mathrm{G}_n(\u^n(t)),\boldsymbol{\varphi}(t)\rangle\d t\right|+ \left|\int_0^T\langle\f_n(t),\boldsymbol{\varphi}(t)\rangle\d t\right|\nonumber\\&\leq C(\|\u_0\|_{\H},\mu ,T,\beta,\|\f_n(t)\|_{\mathrm{L}^2(0,T;\V')})\left[\left(\int_0^T\|\boldsymbol{\varphi}(t)\|_{\V}^2\d t \right)^{1/2}+\left(\int_0^T\|\boldsymbol{\varphi}(t)\|_{\widetilde\L^{r+1}}^{r+1}\d t \right)^{\frac{1}{r+1}}\right]\nonumber\\&\quad +\left(\int_0^T\|\f_n(t)\|_{\V'}^2\d t\right)^{1/2}\left(\int_0^T\|\boldsymbol{\varphi}(t)\|_{\V}^2\d t\right)^{1/2},
	\end{align*}
for all $\boldsymbol{\varphi}\in\mathrm{L}^2(0,T;\V)\cap\mathrm{L}^{r+1}(0,T;\widetilde\L^{r+1}))$.  Thus, it is immediate that
\begin{align} \label{3p16}
	\frac{\d\u^n}{\d t}\xrightarrow{w}\frac{\d\u}{\d t}\ \text{ in } \ \mathrm{L}^2(0,T;\V')+\mathrm{L}^{\frac{r+1}{r}}(0,T;\widetilde\L^{\frac{r+1}{r}}).
\end{align}  Note that $\f^n\to\f$  in $\mathrm{L}^2(0,T;\V')$. 
Since the orthogonal projection $\mathrm{P}_n$ and $\A$ commute, one can easily see for all $\v\in\V$  that 
\begin{align*}
	\|\mathrm{P}_n\v-\v\|_{\V}=\|\A^{1/2}(\mathrm{P}_n\v-\v)\|_{\H}=\|\mathrm{P}_n\A^{1/2}\v-\A^{1/2}\v\|_{\H}\to 0\ \text{ as }\ n\to\infty. 
\end{align*}
  On passing  to limit $n\to\infty$ in (\ref{421}), the limit $\u(\cdot)$ satisfies:
  \begin{equation}\label{eqn43}
		\left\{
		\begin{aligned}
		-\left\langle{\frac{\d\u(t)}{\d t}},\v\right\rangle +\langle\G_0(t),\v\rangle   
		  &=\langle\f(t),\v\rangle,	\\
		\langle\u(0),\v(0)\rangle&=\langle \u_0,\v(0)\rangle,
		\end{aligned}
		\right.
		\end{equation}
	for all $\v\in\V$. In fact, using a density argument, one can also show that \eqref{eqn43} is true for all $\v\in\mathcal{V}_{\infty}$ also, where $\mathcal{V}_{\infty}$ is defined in \eqref{eqn-vinfty}. 
	\vskip 0.1in
	\noindent\textbf{Step (iii):} \emph{Energy equality satisfied by $\u(\cdot)$:} Let us now discuss the energy equality satisfied by $\u(\cdot)$.  It should be noted that such an energy equality is not immediate due to the final convergence in \eqref{3.20}. We follow the approximations given in \cite{CLF} to obtain such an energy equality. In \cite{CLF}, the authors established an approximation of $\u(\cdot)$ in bounded domains such that the approximations are bounded and converge in both Sobolev and Lebesgue spaces simultaneously (one can see \cite{KWH} for such an approximation of $\L^p$-space valued functions using truncated Fourier expansions in periodic domains). We approximate $\u(t),$ for each $t\in[0,T]$, by using the finite-dimensional space spanned by the first $n$ eigenfunctions of the Stokes operator as (\cite[Theorem 4.3]{CLF})
	\begin{align}
		\label{3.32}\u_n(t):=\mathrm{P}_{1/n}\u(t)=\sum_{\lambda_j<n^2}e^{-\lambda_j/n}\langle\u(t),\boldsymbol{w}_j\rangle \boldsymbol{w}_j.
	\end{align}
	For notational convenience, we use the approximations given in \eqref{3.32} as $\u_n(\cdot)$  and the Galerkin approximations  in Steps (i) and (ii) as $\u^n(\cdot)$. Note first that 
	\begin{align}\label{3.36}
	\|\u_n\|_{\H}^2=\|\mathrm{P}_{1/n}\u\|_{\H}^2=\sum_{\lambda_j<n^2}e^{-2\lambda_j/n}|\langle\u,\boldsymbol{w}_j\rangle|^2\leq\sum_{j=1}^{\infty}|\langle\u,\boldsymbol{w}_j\rangle|^2=\|\u\|_{\H}^2<+\infty,
	\end{align}
	for all $\u\in\H$. Moreover, we have 
	\begin{align}\label{3.37}
	\|(\I_d-\mathrm{P}_{1/n})\u\|_{\H}^2&=\|\u\|_{\H}^2-2\langle\u,\mathrm{P}_{1/n}\u\rangle+\|\mathrm{P}_{1/n}\u\|_{\H}^2\nonumber\\&=\sum_{j=1}^{\infty}|\langle\u,\boldsymbol{w}_j\rangle|^2-2\sum_{\lambda_j<n^2}e^{-\lambda_j/n}|\langle\u,\boldsymbol{w}_j\rangle|^2+\sum_{\lambda_j<n^2}e^{-2\lambda_j/n}|\langle\u,\boldsymbol{w}_j\rangle|^2\nonumber\\&=\sum_{\lambda_j<n^2}(1-e^{-\lambda_j/n})^2|\langle\u,\boldsymbol{w}_j\rangle|^2+\sum_{\lambda_j\geq n^2}|\langle\u,\boldsymbol{w}_j\rangle|^2,
	\end{align}
	for all $\u\in\H$. It should be noted that the final term on the right hand side of the equality \eqref{3.37} tends to zero as $n\to\infty$, since the series $\sum_{j=1}^{\infty}|\langle\u,\boldsymbol{w}_j\rangle|^2$ is convergent. The first term on the right hand side of the equality can be bounded from above by $$\sum_{j=1}^{\infty}(1-e^{-\lambda_j/n})^2|\langle\u,\boldsymbol{w}_j\rangle|^2\leq 4\sum_{j=1}^{\infty}|\langle\u,\boldsymbol{w}_j\rangle|^2=4\|\u\|_{\H}^2<+\infty.$$ Using the Dominated Convergence Theorem, one can interchange the limit and sum, and hence we obtain 
	$$\lim_{n\to\infty}\sum_{j=1}^{\infty}(1-e^{-\lambda_j/n})^2|\langle\u,\boldsymbol{w}_j\rangle|^2=\sum_{j=1}^{\infty}\lim_{n\to\infty}(1-e^{-\lambda_j/n})^2|\langle\u,\boldsymbol{w}_j\rangle|^2=0.$$
	Hence $\|(\I_d-\mathrm{P}_{1/n})\u\|_{\H}\to 0$ as $n\to\infty$. 
	Let us now discuss the properties of the approximation given in \eqref{3.32}. The authors in \cite{CLF} showed that such an approximation satisfies the following: 
	\begin{enumerate}
		\item [(1)] $\u_n(t)\in{\mathcal{V}}_{\infty}$ for all $t\in[0,T]$;
		\item [(2)] $\u_n(t)\to\u(t)$ in $\H_0^1(\mathcal{O})$ with $\|\u_n(t)\|_{\H^1}\leq C\|\u(t)\|_{\H^1}$ for all $t\in[0,T]$;
		\item [(3)] $\u_n(t)\to\u(t)$ in $\L^{p}(\mathcal{O})$ with $\|\u_n(t)\|_{{\L}^{p}}\leq C\|\u(t)\|_{{\L}^{p}}$, for any $p\in(1,\infty)$, for all  $t\in[0,T]$;
		\item [(4)] $\u_n(t)$ is divergence free and zero on $\partial\mathcal{O}$, for all  $t\in[0,T]$,
	\end{enumerate}
   It should be noted that for $2\leq d\leq 4$, $\D(\A)\hookrightarrow\H^2(\mathcal{O})\hookrightarrow\L^p(\mathcal{O}),$ for all $p\in(1,\infty)$ (cf. \cite{CLF}). Since $\boldsymbol{w}_j$'s are the eigenfunctions of the Stokes operator $\A$, we get $\boldsymbol{w}_j\in\D(\A)\hookrightarrow\V$ and $\boldsymbol{w}_j\in\D(\A)\hookrightarrow\widetilde{\L}^{r+1}$. We also need the fact \begin{align}\label{335}
 	\|\u_n-\u\|_{\mathrm{L}^{r+1}(0,T;\widetilde{\L}^{r+1})}\to 0, \ \text{ as }\ n\to\infty,
 	\end{align} 
 	which  follows from (2). Since $\u\in\mathrm{L}^{r+1}(0,T;\widetilde{\L}^{r+1})$ and the fact that $\|\u^n(t)-\u(t)\|_{\widetilde{\L}^{r+1}}\to 0$, for all $t\in[0,T]$, one can obtain the above convergence by an application of the dominated convergence theorem (with the dominating function $(1+C)\|\u(t)\|_{\widetilde{\L}^{r+1}}$).  Moreover, using the fact that $\u\in\mathrm{L}^{2}(0,T;\V)$, we also have 
 	\begin{align}\label{335a}
 	\|\u_n-\u\|_{\mathrm{L}^{2}(0,T;\V)}\to 0, \ \text{ as }\ n\to\infty,
 	\end{align}


  In order to prove the energy equality, we follow the works \cite{CLF,GGP,KWH}, etc. Let $\eta(t)$ be an even, positive, smooth function with compact support contained in the interval $(-1, 1)$, such that $\int_{-\infty}^{\infty}\eta(s)\d s=1$. Let us denote by $\eta^h$, a family of mollifiers related to the function $\eta$ as $$\eta^h(s):=h^{-1}\eta(s/h), \ \text{ for } \ h>0.$$ In particular, we get $\int_0^h\eta^h(s)\d s=\frac{1}{2}$. For any function $\v\in\mathrm{L}^p(0, T; \X)$, where $\X$ is a Banach space, $p\in[1,\infty)$, we define its mollification in time by $\v^h(\cdot)$ as 
 \begin{align*}
 \v^h(s):=(\v*\eta^h)(s)=\int_0^T\v(\tau)\eta^h(s-\tau)\d\tau, \ \text{ for }\ h\in(0,T). 
\end{align*}
From  \cite[Lemma 2.5, pp. 16]{GGP}, we know that this mollification has the following properties.  For any $\v\in\mathrm{L}^p(0, T; \X)$, $\v^h\in\C^k([0,T);\X)$ for all $k\geq 0$ and  $$\lim_{h\to 0}\|\v^h-\v\|_{\mathrm{L}^p(0,T;\X)}=0.$$ Furthermore, if $\{\v_n\}_{n=1}^{\infty}$ converges to $\v$ in $\mathrm{L}^p(0,T;\X)$, then $$\lim_{n\to \infty}\|\v^h_n-\v^h\|_{\mathrm{L}^p(0,T;\X)}=0.$$   

Let $\{\u_n\}_{n\geq1}$ be a sequence of test functions in $\mathcal{V}_T$ converging to a weak solution $\u$ of \eqref{kvf}. As in \cite{CLF,KWH}, for any fixed time $t_1\in(0,T)$, choose a sequence of test functions 
 \begin{align*}
 \u_n^h(t):=\int_0^{t_1}\eta^h(t-s)\u_n(s)\d s,
 \end{align*}
with the parameter $h$ satisfying  $0<h<T-t_1$ and $h<t_1,$  where $\eta_h$ is the even mollifier given above. Since $\u_n^h\in{\mathcal{V}}_{\infty}$, one can choose $\u_n^h$ as test function in \eqref{eqn43} to find 
 \begin{align}\label{3.23}
 -\int_0^{t_1}\bigg\langle\u(s),\frac{\d\u_n^h(s)}{\d t}\bigg\rangle\d s+ \int_0^{t_1}\langle\G_0(s)-\f(s),\u_n^h(s)\rangle\d s=-\langle\u(t),\u_n^h(t)\rangle +\langle\u(0),\u_n^h(0)\rangle.
 \end{align}
 Note that 
 \begin{align*}
 	\left\|\frac{\d\u_n^h(s)}{\d t}-\frac{\d\u^h(s)}{\d t}\right\|_{\H}&=\left\|\int_0^{t_1}\dot{\eta}^h(t-s)(\u_n^h(s)-\u_n(s))\d s\right\|_{\H}\nonumber\\&\leq\int_0^{t_1}\dot{\eta}^h(t-s)\|\u_n^h(s)-\u_n(s)\|_{\H}\d s\ \to 0\ \text{ as }\ n\to\infty,
 \end{align*}
where $\dot{\eta}^h$ derivative of $\eta^h$ with respect to $t$, so that 
\begin{align}\label{3.24}
		-\int_0^{t_1}\bigg\langle\u(s),\frac{\d\u_n^h(s)}{\d t}\bigg\rangle\d s\to 	-\int_0^{t_1}\bigg\langle\u(s),\frac{\d\u^h(s)}{\d t}\bigg\rangle\d s\ \text{ as } \ n\to\infty. 
\end{align}
 For $\G_0=\G_0^1+\G_0^2$ with $\G_0^1\in\mathrm{L}^2(0,T;\V')$ and $\G_0^2\in\mathrm{L}^{\frac{r+1}{r}}(0,T;\wi\L^{\frac{r+1}{r}})$, we have 
 \begin{align}\label{3.25}
& \left|\int_0^{t_1}\langle\G_0(s),\u_n^h(s)-\u^h(s)\rangle\d s\right|\nonumber\\&\leq \left|\int_0^{t_1}\langle\G_0^1(s),\u_n^h(s)-\u^h(s)\rangle\d s\right|+\left|\int_0^{t_1}\langle\G_0^2(s),\u_n^h(s)-\u^h(s)\rangle\d s\right|\nonumber\\&\leq \int_0^{t_1}\|\G_0^1(s)\|_{\V'}\|\u_n^h(s)-\u^h(s)\|_{\V}\d s+\int_0^{t_1}\|\G_0^2(s)\|_{\widetilde\L^{\frac{r+1}{r}}}\|\u_n^h(s)-\u^h(s)\|_{\widetilde\L^{r+1}}\d s\nonumber\\&\leq \left(\int_0^{t_1}\|\G_0^1(s)\|_{\V'}^2\d s\right)^{1/2} \left(\int_0^{t_1}\|\u_n^h(s)-\u^h(s)\|_{\V}^2\d s\right)^{1/2} \nonumber\\&\quad+\left(\int_0^{t_1}\|\G_0^2(s)\|_{\widetilde\L^{\frac{r+1}{r}}}^{\frac{r+1}{r}}\right)^{\frac{r}{r+1}}\left(\int_0^{t_1}\|\u_n^h(s)-\u^h(s)\|_{\widetilde\L^{r+1}}^{r+1}\right)^{\frac{1}{r+1}}\to 0 \ \text{ as }\ n\to\infty. 
 \end{align}
 Moreover, we get 
 \begin{align}\label{325}
 \left|\int_0^{t_1}\langle\f(s),\u_n^h(s)-\u^h(s)\rangle \d s\right|&\leq\int_0^{t_1}\|\f(s)\|_{\V'}\|\u_n^h(s)-\u^h(s)\|_{\V}\d s\nonumber\\&\leq \left(\int_0^{t_1}\|\f(s)\|_{\V'}^2\d s\right)^{1/2}\left(\int_0^{t_1}\|\u_n^h(s)-\u^h(s)\|_{\V}^2\d s\right)^{1/2}\nonumber\\&\to 0 \ \text{ as }\ n\to\infty. 
 \end{align}
 Using the convergences \eqref{3.24}-\eqref{325} in \eqref{3.23}, we deduce 
  \begin{align}\label{3.26}
 -\int_0^{t_1}\bigg\langle\u(s),\frac{\d\u^h(s)}{\d t}\bigg\rangle\d s+ \int_0^{t_1}\langle\G_0(s)-\f(s),\u^h(s)\rangle\d s= -\langle\u(t),\u^h(t)\rangle +\langle\u(0),\u^h(0)\rangle.
 \end{align}
 Since the function $\eta^h$ is even in $(-h, h)$, we have  $\dot{\eta}^h(r)=-\dot{\eta}^h(-r)$ and hence (see \cite[pp. 7154]{KWH}) 
 \begin{align}
 \int_0^{t_1}\bigg\langle\u(s),\frac{\d\u^h(s)}{\d t}\bigg\rangle\d s= \int_0^{t_1}\int_0^{t_1}\dot{\eta}^h(t-s)\langle\u(t),\u(s)\rangle\d t\d s=0. 
 \end{align}
 Thus, from \eqref{3.26}, it is immediate that 
 \begin{align}
 \int_0^{t_1}\langle\G_0(s)-\f(s),\u^h(s)\rangle\d s=-\langle\u(t_1),\u^h(t_1)\rangle +\langle\u(0),\u^h(0)\rangle.
 \end{align}
 Now, letting $h\to 0$, and arguing similarly as in \eqref{3.25} and \eqref{325}, we find 
 \begin{align*}
 \lim_{h\to 0} \int_0^{t_1}\langle\G_0(s),\u^h(s)\rangle\d s= \int_0^{t_1}\langle\G_0(s),\u(s)\rangle\d s,
 \end{align*}
 and 
  \begin{align*}
 \lim_{h\to 0} \int_0^{t_1}\langle\f(s),\u^h(s)\rangle\d s= \int_0^{t_1}\langle\f(s),\u(s)\rangle\d s.
 \end{align*}
 The above convergence gives us 
 \begin{align}
 \int_0^{t_1}\langle\G_0(s)-\f(s),\u(s)\rangle\d s=-\lim_{h\to 0}\langle\u(t_1),\u^h(t_1)\rangle +\lim_{h\to 0}\langle\u(0),\u^h(0)\rangle.
 \end{align}
 Using the fact that $\u$ is $\mathrm{L}^2$-weakly continuous in time and $\int_0^h\eta^h(s)\d s=\frac{1}{2}$, we get 
 \begin{align}
 \langle\u(t_1),\u^h(t_1)\rangle &=\int_0^{t_1}\eta^h(s)\langle\u(t_1),\u(t_1-s)\rangle\d s\nonumber\\&=\frac{1}{2}\|\u(t_1)\|_{\H}^2+\int_0^h\eta^h(s)\langle\u(t_1),\u(t_1-s)-\u(t_1)\rangle\d s\to \frac{1}{2}\|\u(t_1)\|_{\H}^2,
 \end{align}
as $h\to 0$. Similarly, we obtain 
\begin{align}
\langle\u(0),\u^h(0)\rangle\to\frac{1}{2}\|\u(0)\|_{\H}^2\ \text{ as } \ h\to 0. 
\end{align}
Combining the above convergences, we finally obtain the energy equality 
\begin{align}
\frac{1}{2}\|\u(t_1)\|_{\H}^2+\int_0^{t_1}\langle\G_0(s)-\f(s),\u(s)\rangle\d s=\frac{1}{2}\|\u(0)\|_{\H}^2,
\end{align}
for all $t_1\in(0,T)$. Thus the system (\ref{eqn43}) satisfies the following energy equality:
	\begin{align}\label{eqn44}
	\|\u(t)\|_{\H}^2+2\int_0^t\langle \G_0(s)-\f(s),\u(s)\rangle \d
	s=\|\u(0)\|_{\H}^2,
	\end{align}
	for any $t\in (0,T)$. Moreover, we have the following equality:
	\begin{align}\label{eqn45}
	&e^{-2\varrho t}\|\u(t)\|_{\H}^2+2\int_0^te^{-2\varrho t}\langle \G_0(s)-\f(s)+\varrho\u(s),\u(s)\rangle \d
	s=\|\u(0)\|_{\H}^2,
	\end{align}
	for all $t\in(0,T]$.

\vskip 0.1in
\noindent\textbf{Step (iv).} \emph{Minty-Browder technique:} Remember that $\u^n(0)=\mathrm{P}_n\u(0)$, and hence the initial	value $\u^n(0)$ converges to $\u(0)$ strongly in $\H$, that is, we have 
\begin{align}\label{eqn46}
\lim_{n\to\infty}\|\u^n(0)-\u(0)\|_{\H}=0.
\end{align}	For any
$\v\in\mathrm{L}^{\infty}(0,T;\H_m)$ with
$m<n$, using the monotonicity property (see \eqref{fe}), we obtain 
\begin{align}\label{eqn48}
&\int_0^{T}e^{-2\varrho t}\left\{\langle \G(\v(t))-\G(\u^n(t)),\v(t)-\u^n(t)\rangle 
+\varrho\left(\v(t)-\u^n(t),\v(t)-\u^n(t)\right)\right\}\d
t\geq 0.
\end{align}
Applying the energy equality (\ref{eqn42}) in \eqref{eqn48}, we find
\begin{align}\label{eqn49}
&\int_0^{T}e^{-2\varrho t}\langle \G(\v(t))+\varrho\v(t),\v(t)-\u^n(t)\rangle \d
t\nonumber\\&\geq
\int_0^{T}e^{-2\varrho t}\langle \G(\u^n(t))+\varrho\u^n(t),\v(t)-\u^n(t)\rangle \d
t\nonumber\\&=\int_0^{T}e^{-2\varrho t}
\langle \G(\u^n(t))+\varrho\u^n(t),\v(t)\rangle \d t
\nonumber\\&\quad+\frac{1}{2}\Big(e^{-2\varrho T}\|\u^n(T)\|_{\H}^2-\|\u^n(0)\|_{\H}^2\Big)
-\int_0^{T}e^{-2\varrho t}\langle \f^n(t),\u^n(t)\rangle \d t.
\end{align}
Taking limit infimum on both sides of (\ref{eqn49}), we further have 
\begin{align}\label{eqn52}
&\int_0^{T}e^{-2\varrho t}\langle \G(\v(t))+\varrho\v(t),\v(t)-\u(t)\rangle \d
t\nonumber\\&\geq \int_0^{T}e^{-2\varrho t}
\langle \G_0(t)+\varrho\u(t),\v(t)\rangle \d
t\nonumber\\&\quad+\frac{1}{2}\liminf_{n\to\infty}\Big(e^{-2\varrho T}\|\u^n(T)\|_{\H}^2-\|\u^n(0)\|_{\H}^2\Big)
-\int_0^{T}e^{-2\varrho t}\langle \f(t),\u(t)\rangle \d t.
\end{align}
Making use of the lower semicontinuity property of the $\H$-norm, and the strong convergence of the initial data (see \eqref{eqn46}), we get 
\begin{align}\label{eqn53}
&\liminf_{n\to\infty}\Big\{e^{-2\varrho T}\|\u^n(T)\|_{\H}^2-\|\u^n(0)\|_{\H}^2\Big\}\geq
e^{-2\varrho T}\|\u(T)\|^2_{\H}-\|\u(0)\|^2_{\H}.
\end{align}
Applying  (\ref{eqn53}) and the energy equality (\ref{eqn45}) in (\ref{eqn52}), we deduce that 
\begin{align}\label{eqn55}
&\int_0^{T}e^{-2\varrho t}\langle \G(\v(t))+\varrho\v(t),\v(t)-\u(t)\rangle \d
t\nonumber\\&\geq \int_0^{T}e^{-2\varrho t}
\langle \G_0(t)+\varrho\u(t),\v(t)\rangle \d
t+\frac{1}{2}\Big(	e^{-2\varrho T}\|\u(T)\|^2_{\H}-\|\u(0)\|^2_{\H}\Big)
-\int_0^{T}e^{-2\varrho t}\langle \f(t),\u(t)\rangle \d
t\nonumber\\&=\int_0^{T}e^{-2\varrho t}
\langle \G_0(t)+\varrho\u(t),\v(t)\rangle \d
t-\int_0^{T}e^{-2\varrho t}\langle \G_0(t)+\varrho\u(t),\u(t)\rangle \d
t\nonumber\\&=\int_0^{T}e^{-2\varrho t}
\langle \G_0(t)+\varrho\u(t),\v(t)-\u(t)\rangle \d t.
\end{align}
Note that the estimate (\ref{eqn55}) holds true for any
$\v\in\mathrm{L}^{\infty}(0,T;\H_m)$, $m\in\mathbb{N}$, since the  inequality given in (\ref{eqn55}) is
independent of both $m$ and $n$. Using a density argument, one can show that the inequality (\ref{eqn55}) remains true for any
$$\v\in\mathrm{L}^{\infty}(0,T;\H)\cap\mathrm{L}^2(0,T;\V)\cap\mathrm{L}^{r+1}(0, T ; \widetilde\L^{r+1})).$$ In fact, for any
$\v\in\mathrm{L}^{\infty}(0,T;\H)\cap\mathrm{L}^2(0,T;\V)\cap\mathrm{L}^{r+1}(0, T ; \widetilde\L^{r+1})),$ there	exists a strongly convergent subsequence
$\v_m\in\mathrm{L}^{\infty}(0,T;\H)\cap\mathrm{L}^2(0,T;\V)\cap\mathrm{L}^{r+1}(0, T ; \widetilde\L^{r+1})),$ that
satisfies the inequality in  (\ref{eqn55}).

Taking $\v=\u+\lambda\w$, $\lambda>0$, where
$\w\in\mathrm{L}^{\infty}(0,T;\H)\cap\mathrm{L}^2(0,T;\V)\cap\mathrm{L}^{r+1}(0, T ; \widetilde\L^{r+1})),$ and
substituting for $\v$ in (\ref{eqn55}), we obtain 
\begin{align}\label{eqn56}
\int_0^{T}e^{-2\varrho t}\langle \G(\u(t)+\lambda\w(t))-\G_0(t)+\varrho\lambda\w(t),\lambda\w(t)\rangle \d
t\geq 0.
\end{align}
Dividing the inequality in (\ref{eqn56}) by $\lambda$, using the
hemicontinuity property of the operator $\G(\cdot)$ (see Lemma \ref{lem2.8}), and then passing to the limit with $\lambda\to 0$, we get 
\begin{align}\label{eqn60}
\int_0^{T}e^{-2\varrho t}\langle \G(\u(t))-\G_0(t),\w(t)\rangle \d
t\geq 0,
\end{align}
for any
$\w\in\mathrm{L}^{\infty}(0,T;\H)\cap\mathrm{L}^2(0,T;\V)\cap\mathrm{L}^{r+1}(0, T ; \widetilde\L^{r+1})).$
It should be noted that the term
${\int_0^{T}}e^{-2\varrho t} \varrho\lambda\langle\w(t),\w(t)\rangle \d
t$ in \eqref{eqn56} tends to $0$ as $\lambda\to0$. Now, changing $\w$ to $-\w$ in \eqref{eqn60}, we obtain
\begin{align}\label{eqn600}
	\int_0^{T} e^{-2\varrho t}\langle \G(\u(t))-\G_0(t),\w(t)\rangle \d t\leq  0,
\end{align}
so that 
\begin{align}\label{eqn-6000}
	\int_0^{T} e^{-2\varrho t}\langle \G(\u(t))-\G_0(t),\w(t)\rangle \d t=  0,
\end{align}
for any
$\w\in\mathrm{L}^{\infty}(0,T;\H)\cap\mathrm{L}^2(0,T;\V)\cap\mathrm{L}^{r+1}(0,T;\widetilde\L^{r+1})).$ Since $e^{-2\varrho T}\leq e^{-2\varrho t}\leq 1$ and $\w$ is arbitrary, we deduce 
$$\G(\u(t))=\G_0(t), \ \text{ for a.e.}\ t\in[0,T].$$
 One can easily see the following energy equality satisfied by $\u(\cdot)$, by making use of \eqref{eqn44}: 
\begin{align}\label{3.47}
\|\u(t)\|_{\H}^2+2\mu\int_0^t\|\u(s)\|_{\V}^2\d s+ 2\beta\int_0^t\|\u(s)\|_{\widetilde\L^{r+1}}^{r+1}\d s=\|\u_0\|_{\H}^2+\int_0^t\langle\f(s),\u(s)\rangle\d s,
\end{align}
for all $t\in[0,T]$. Moreover, $\u(\cdot)$ satisfies 
\begin{align}\label{energy3}
\sup_{t\in[0,T]}\|\u(t)\|_{\H}^2+\mu \int_0^T\|\u(t)\|_{\V}^2\d t+2\beta\int_0^T\|\u(t)\|_{\widetilde\L^{r+1}}^{r+1}\d t\leq \|\u_0\|_{\H}^2+\frac{1}{4\mu }\int_0^T\|\f(t)\|_{\V'}^2\d t.
\end{align}
Hence, we get  $$\u\in\mathrm{L}^{\infty}(0,T;\H)\cap\mathrm{L}^2(0,T;\V)\cap\mathrm{L}^{r+1}(0, T ; \widetilde\L^{r+1})),$$ with $\partial_t\u\in\mathrm{L}^2(0,T;\V')+\mathrm{L}^{\frac{r+1}{r}}(0,T;\widetilde\L^{\frac{r+1}{r}})$.

Finally, the continuity of $\u:[0,T]\to\H$, as discussed in \cite{CLF,KWH}, follows by combining the weak continuity of $\u$ into $\H$ and the continuity of $t\mapsto\|\u(t)\|_{\H}^2$, which is a consequence of the energy equality \eqref{3.47}.

\vskip 0.2cm

\noindent\textbf{Part (2). Uniqueness:} Let $\u_1(\cdot)$ and $\u_2(\cdot)$ be two weak solutions of the system (\ref{kvf}) satisfying \eqref{energy3}. Then $\u_1(\cdot)-\u_2(\cdot)$ satisfies: 	
\begin{align}\label{346}
&\|\u_1(t)-\u_2(t)\|_{\H}^2+2\mu \int_0^t\|\u_1(s)-\u_2(s)\|_{\V}^2\d s\nonumber\\&=\|\u_1(0)-\u_2(0)\|_{\H}^2-2\int_0^t\langle\B(\u_1(s))-\mathrm{B}(\u_2(s)),\u_1(s)-\u_2(s)\rangle\d s\nonumber\\&\quad -2\beta\int_0^t\langle\mathcal{C}(\u_1(s))-\mathcal{C}_2(\u_2(s)),\u_1(s)-\u_2(s)\rangle\d s.
\end{align}
Note that $\langle\B(\u_1)-\B(\u_2),\u_1-\u_2\rangle=\langle\B(\u_1-\u_2,\u_2),\u_1-\u_2\rangle,$ since $\langle\B(\u_1,\u_1-\u_2),\u_1-\u_2\rangle=0$. 
An estimate similar to \eqref{2.30} yields
\begin{align}\label{347}
&|\langle\B(\u_1-\u_2,\u_2),\u_1-\u_2\rangle|\leq\frac{\mu }{2}\|\u_1-\u_2\|_{\V}^2+\frac{\beta}{2}\||\u_2|^{\frac{r-1}{2}}(\u_1-\u_2)\|_{\H}^2+\varrho\|\u_1-\u_2\|_{\H}^2,
\end{align}
for $r>3$. A calculation similar to \eqref{2.23} gives 
\begin{align}\label{261}
&\beta\langle\mathcal{C}(\u_1)-\mathcal{C}(\u_2),\u_1-\u_2\rangle\geq \frac{\beta}{2}\||\u_2|^{\frac{r-1}{2}}(\u_1-\u_2)\|_{\H}^2.
\end{align}
Using (\ref{347}) and (\ref{261}) in \eqref{346}, we arrive at 
\begin{align}\label{3.18}
&\|\u_1(t)-\u_2(t)\|_{\H}^2+\mu \int_0^t\|\u_1(s)-\u_2(s)\|_{\V}^2\d s\nonumber\\&\leq\|\u_1(0)-\u_2(0)\|_{\H}^2+2\varrho\int_0^t\|\u_1(s)-\u_2(s)\|_{\H}^2\d s.
\end{align}
Applying Gronwall's inequality in (\ref{3.18}), we get  
\begin{align}\label{269}
\|\u_1(t)-\u_2(t)\|_{\H}^2&\leq \|\u_1(0)-\u_2(0)\|_{\H}^2e^{2\varrho T},
\end{align}
and hence the uniqueness follows by taking $\u_1(0)=\u_2(0)$ in \eqref{269}.
\end{proof}
Let us now sketch the proof of the existence and uniqueness of weak solution for the system \eqref{kvf} with $r=3,$ for $2\beta\mu \geq 1$. This means that when both the viscosity of a fluid and the porosity of a porous medium are sufficiently large, then the corresponding flow stays bounded and regular. When the viscosity is small, one can still obtain a unique weak solution by letting the porosity to be large and vice versa. In fact, in Theorem \ref{main3} below, we show that a unique weak solution exists for any $\beta,\mu>0$. But the continuous dependence on the initial data is till known to exists for $2\beta\mu \geq 1$ (see Remark \ref{rem-con} below).

\begin{theorem}\label{main1}
	Let $\u_0\in \H$ and $\f\in\mathrm{L}^{2}(0,T;\V')$  be given.  Then, for $2\beta\mu \geq 1$, there exists a unique weak solution to the system (\ref{kvf}) satisfying $$\u\in\mathrm{C}([0,T];\H)\cap\mathrm{L}^2(0,T;\V)\cap\mathrm{L}^{4}(0,T;\L^{4}(\mathcal{O})).$$ 
\end{theorem}
\begin{proof}
One can follow in a similar way as in the proof of Theorem \ref{main}. In the critical case $r=3$, note that the operator $\G(\cdot)$ is monotone (see \eqref{218}), and hence we provide a short proof in this case. Using the convergences given in \eqref{3.20}, we take limit supremum in \eqref{422} to find 
\begin{align}\label{263}
&\limsup_{n\to\infty}\int_0^t(\G(\u^n(s)),\u^n(s))\d s\nonumber\\&=\limsup_{n\to\infty}\left\{\frac{1}{2}\left[\|\u^n(0)\|_{\H}^2-\|\u^n(t)\|_{\H}^2\right]+\int_0^t(\f^n(s),\u^n(s))\d s\right\}\nonumber\\&\leq \frac{1}{2}\left[\limsup_{n\to\infty}\|\u^n(0)\|_{\H}^2-\liminf_{n\to\infty}\|\u^n(t)\|_{\H}^2\right]+\limsup_{n\to\infty}\int_0^t(\f^n(s),\u^n(s))\d s\nonumber\\&\leq\frac{1}{2}\left[\|\u_0\|_{\H}^2-\|\u(t)\|_{\H}^2\right]+\int_0^t\langle\f(s),\u(s)\rangle\d s=\int_0^t\langle \G_0(s),\u(s)\rangle \d
s,
\end{align}
for all $t\in(0,T]$, where we used \eqref{eqn44}, the weak lower-semicontinuity property of the $\H$-norm, and strong convergence of the initial data and forcing. Since $\G(\cdot)$ is a monotone  operator, from \eqref{218}, we also have 
\begin{align*}
\int_0^t\langle\G(\u^n(s))-\G(\v(s)),\u^n(s)-\v(s)\rangle\d s\geq 0. 
\end{align*}
Taking limit supremum on both sides and using \eqref{263}, we further find 
\begin{align}
\int_0^t\langle \G_0(s)-\G(\v(s)),\u(s)-\v(s)\rangle \d
s\geq 0. 
\end{align}
Taking $\u-\v=\lambda\w$, for $\lambda>0$, dividing by $\lambda$, passing to the limit with $\lambda\to 0$, and then exploiting the hemicontinuity property of the operator $\G(\cdot)$, we finally obtain $\G_0(t)=\G(\u(t))$. Uniqueness follows by taking the estimate \eqref{232} instead of \eqref{347} and following similarly as in the Part (2) of Theorem \ref{main}. 
\end{proof}

\begin{remark}
	We point out here  that the results obtained in Theorems \ref{main} and \ref{main1} hold true in unbounded domains also as we are not using any compactness arguments to get the required results. As discussed in \cite[Section 2.5]{KKMTM}, in the unbounded domain case, one needs to replace the eigenfunctions of the Stokes operator  in Step (iii) in the proof of Theorem \ref{main} by the eigenfunctions of the operator defined in  \cite[Section 2.5]{KKMTM}. 
\end{remark}

\begin{theorem}\label{main2}
	Let $\u_0\in \H$ and $\f\in\mathrm{L}^{2}(0,T;\V')$  be given. For $2\leq d\leq4$ and $r\in[1,3]$,   there exists a weak solution to the system (\ref{kvf}) with $\u\in\mathrm{L}^{\infty}(0,T;\H)\cap\mathrm{L}^2(0,T;\V)$ satisfying the following energy inequality: 
	\begin{align}\label{ener-ine}
		\|\u(t)\|_{\H}^2+2\mu\int_0^t\|\u(s)\|_{\V}^2\d s+ 2\beta\int_0^t\|\u(s)\|_{\widetilde\L^{r+1}}^{r+1}\d s\leq \|\u_0\|_{\H}^2+\int_0^t\langle\f(s),\u(s)\rangle\d s,
	\end{align}
	for all $t\in[0,T]$, and the initial data is satisfied in the following sense:
	\begin{align}\label{360}
		\lim_{t\to 0}\|\u(t)-\u_0\|_{\H}=0.
	\end{align}
	
	  Moreover, for $d=2$ and $r\in[1,3]$, the energy inequality becomes equality and the weak solution is unique. 
\end{theorem}
\begin{proof}
For the case $2\leq d\leq4$ and $r\in[1,3]$, one can obtain the existence of a weak solution by using compactness arguments.  Using the convergences given in \eqref{3.20} and \eqref{3p16}, and the fact that the embedding $\V\hookrightarrow\H$ is compact,  an application of Aubin-Lions compactness lemma yields 
\begin{align}\label{3p58}
		\u^n&\rightarrow\u\ \text{ in }\ \mathrm{L}^{2}(0,T;\H),
\end{align}
along a further subsequence still denoted by the same symbol. The above convergence  implies further along a subsequence (still denoted by the same symbol) that 
\begin{align}\label{3p61}
	\u^n(t,x)\to \u(t,x)\ \text{ for a.e. }\ (t,x)\in(0,T)\times\Omega. 
\end{align}
The weak convergence 	$\u^n\xrightarrow{w}\u \text{ in } \mathrm{L}^{2}(0,T;\V)$ and the energy estimate \eqref{energy1} imply that 
\begin{align}\label{3p60}
&	\left|\int_0^T\langle\A_n\u_n(t),\mathrm{P}_n\v(t)\rangle\d t-\int_0^T\langle\A\u(t),\v(t)\rangle\d t\right|\nonumber\\&\leq\left|\int_0^T\langle\A\u_n(t),\mathrm{P}_n\v(t)-\v(t)\rangle\d t\right|+\left|\int_0^T\langle\A\u_n(t)-\A\u(t),\v(t)\rangle\d t\right|\nonumber\\&\leq\|\mathrm{P}_n-\I_d\|_{\mathcal{L}(\V)}\left(\int_0^T\|\u_n(t)\|_{\V}^2\d t\right)^{1/2}\left(\int_0^T\|\v(t)\|_{\V}^2\d t\right)^{1/2}+\left|\int_0^T\langle\A\u_n(t)-\A\u(t),\v(t)\rangle\d t\right|\nonumber\\&\to\ \text{ as }\ n\to\infty, 
\end{align}
for all $\v\in\mathrm{L}^{2}(0,T;\V)$. For $\boldsymbol{\psi}\in \C([0,T];\V)$, using Ladyzhenskaya's and H\"older's inequalities, the convergence \eqref{3p58} and the estimate \eqref{energy1},   we have 
\begin{align}
	&	\left|\int_0^T\langle\B_n(\u_n(t)),\boldsymbol{\psi}(t)\rangle\d t-\int_0^T\langle\B(\u(t)),\boldsymbol{\psi}(t)\rangle\d t\right|
		\nonumber\\&\leq\left|\int_0^T\langle\B(\u_n(t)-\u(t),\u_n(t)),\boldsymbol{\psi}(t)\rangle\d t\right|+\left|\int_0^T\langle\B(\u(t),\u_n(t)-\u(t)),\boldsymbol{\psi}(t)\rangle\d t\right| 
	\nonumber\\&=\left|\int_0^T\langle\B(\u_n(t)-\u(t),\boldsymbol{\psi}(t)),\u_n(t)\rangle\d t\right|+\left|\int_0^T\langle\B(\u(t),\boldsymbol{\psi}(t)),\u_n(t)-\u(t)\rangle\d t\right| \nonumber\\&\leq CT^{\frac{4-d}{8}}\sup_{t\in[0,T]}\|\boldsymbol{\psi}(t)\|_{\V}\left(\sup_{t\in[0,T]}\left[\|\u_n(t)\|_{\H}^{\frac{4-d}{4}}+\|\u_n(t)\|_{\H}^{\frac{4-d}{4}}\right]\right)\nonumber\\&\qquad\times\left[\int_0^T\left(\|\u_n(t)\|_{\V}^2+\|\u(t)\|_{\V}^2\right)\d t\right]^{\frac{d}{4}}\left(\int_0^T\|\u_n(t)-u(t)\|_{\H}^2\d t\right)^{\frac{4-d}{8}}\nonumber\\&\to 0\ \text{ as }\ n\to\infty. 
\end{align}
Since $\mathrm{L}^{\frac{4}{4-d}}(0,T;\V)$ is dense in $\C([0,T];\V)$, for any given $\v\in\mathrm{L}^{\frac{4}{4-d}}(0,T;\V)$ and $\var>0$,  there exists a sequence $\boldsymbol{\psi}_{\varepsilon}\in\C([0,T];\V)$ such that $$\|\boldsymbol{\psi}_{\varepsilon}-\v\|_{\mathrm{L}^{\frac{4}{4-d}}(0,T;\V)}\leq \var.$$ Let us now consider  
\begin{align*}
	&	\left|\int_0^T\langle\B_n(\u_n(t)),\mathrm{P}_n\v(t)\rangle\d t-\int_0^T\langle\B(\u(t)),\v(t)\rangle\d t\right|\nonumber\\&\leq\left|\int_0^T\langle\B(\u_n(t)),\mathrm{P}_n\v(t)-\v(t)\rangle\d t\right|+\left|\int_0^T\langle\B(\u_n(t))-\B(\u(t)),\v(t)-\boldsymbol{\psi}_{\varepsilon}(t)\rangle\d t\right|\nonumber\\&\quad+\left|\int_0^T\langle\B(\u_n(t))-\B(\u(t)),\boldsymbol{\psi}_{\varepsilon}(t)\rangle\d t\right| \nonumber\\&\leq\|\mathrm{P}_n-\I_d\|_{\mathcal{L}(\V)}\left(\int_0^T\|\u_n(t)\|_{\wi\L^{4}}^{\frac{8}{d}}\d t\right)^{\frac{d}{4}}\left(\int_0^T\|\v(t)\|_{\V}^{\frac{4}{4-d}}\d t\right)^{\frac{4-d}{4}}\nonumber\\&\quad+C\left[\int_0^T\left(\|\u_n(t)\|_{\wi\L^4}^{\frac{8}{d}}+\|\u(t)\|_{\wi\L^4}^{\frac{8}{d}}\right)\d t\right]^{\frac{d}{4}}\left(\int_0^T\|\v(t)-\boldsymbol{\psi}_{\varepsilon}(t)\|_{\V}^{\frac{4}{4-d}}\d t\right)^{\frac{4-d}{4}}\nonumber\\&\quad+\left|\int_0^T\langle\B(\u_n(t))-\B(\u(t)),\boldsymbol{\psi}_{\varepsilon}(t)\rangle\d t\right|
	\nonumber\\&\leq C\left[\int_0^T\left(\|\u_n(t)\|_{\wi\L^4}^{\frac{8}{d}}+\|\u(t)\|_{\wi\L^4}^{\frac{8}{d}}\right)\d t\right]^{\frac{d}{4}}\left(\int_0^T\|\v(t)-\boldsymbol{\psi}_{\varepsilon}(t)\|_{\V}^{\frac{4}{4-d}}\d t\right)^{\frac{4-d}{4}}\  \text{ as } n\to\infty,
\end{align*}
Since an application of Ladyzheskaya's inequality and the energy estimate \eqref{energy1} yield
\begin{align*}
	\int_0^T\|\u_n(t)\|_{\wi\L^4}^{\frac{8}{d}}\d t\leq C\sup_{t\in[0,T]}\|\u_n(t)\|_{\H}^{\frac{2(4-d)}{d}}\int_0^T\|\u_n(t)\|_{\V}^2\d t\leq \left\{ \|\u_0\|_{\H}^2+\frac{1}{\mu }\int_0^T\|\f(t)\|_{\V'}^2\d t\right\}^{\frac{4}{d}},
\end{align*}
so that 
\begin{align*}
	&	\left|\int_0^T\langle\B_n(\u_n(t)),\mathrm{P}_n\v(t)\rangle\d t-\int_0^T\langle\B(\u(t)),\v(t)\rangle\d t\right|\nonumber\\&\leq C\left\{ \|\u_0\|_{\H}^2+\frac{1}{\mu }\int_0^T\|\f(t)\|_{\V'}^2\d t\right\}\var.
	\end{align*}
	Since $\var>0$ is arbitrary, we deduce that 
	\begin{align}\label{3p62}
		\lim_{n\to\infty}\int_0^T\langle\B_n(\u_n(t)),\mathrm{P}_n\v(t)\rangle\d t=\int_0^T\langle\B(\u(t)),\v(t)\rangle\d t. 
	\end{align}
We know from \eqref{energy1} that 
\begin{align}\label{3p57}
	\int_0^T	\|\mathcal{C}_n(\u_n(t)\|_{\wi\L^{\frac{r+1}{r}}}^{\frac{r+1}{r}}dt\leq \int_0^T\|\u_n(t)\|_{\wi\L^{r+1}}^{r+1}\d t\leq  \|\u_0\|_{\H}^2+\frac{1}{\mu }\int_0^T\|\f(t)\|_{\V'}^2\d t.
\end{align}
Therefore, by an application of the Banach-Alaoglu theorem yields the existence of a $\boldsymbol{\xi}\in\mathrm{L}^{\frac{r+1}{r}}(0,T;\wi\L^{\frac{r+1}{r}})$ such that 
\begin{align}
	\mathcal{C}_n(\u_n)\xrightarrow{w} \boldsymbol{\xi}\ \text{ in }\ \mathrm{L}^{\frac{r+1}{r}}(0,T;\wi\L^{\frac{r+1}{r}}) .
\end{align}
Note that from \eqref{3p57}, we have $\{|\u_n(t)|^{r-1}\u_n(t)\}_{n\in\N}, |\u(t)|^{r-1}\u(t) \in \mathrm{L}^{\frac{r+1}{r}}((0,T)\times \Omega)$ and 
\[\|{ |\u_n(t)|^{r-1}\u_n(t)}\|_{\mathrm{L}^{\frac{r+1}{r}}((0,T)\times \Omega)} \leq    \left(\|\u_0\|_{\H}^2+\frac{1}{\mu }\int_0^T\|\f(t)\|_{\V'}^2\d t\right)^{\frac{1}{r+1}},\]
and \eqref{3p61} implies $|\u_{n}(t)|^{r-1}\u_{n}(t) \to |\u(t)|^{r-1}\u(t)$ a.e. in $(0,T)\times \Omega$. According to \cite[Lemma 1.3, Chapter 1, pp. 12]{JLF}, we deduce
\begin{align*}
	|\u_{n}(t)|^{r-1}\u_{n}(t) \xrightarrow{w} |\u(t)|^{r-1}\u(t) \  \mbox{ in } \  \mathrm{L}^{\frac{r+1}{r}}((0,T)\times \Omega).
\end{align*} 
Thus, by the uniqueness of weak limits, we infer that $\boldsymbol{\xi}=\mathcal{P} (|\u|^{r-1}\u)$. It implies that
\begin{align}\label{3p65}
	\mathcal{C}_n(\u_{n}) \xrightarrow{w} \mathcal{C}(\u) \  \mbox{ in } \  \mathrm{L}^{\frac{r+1}{r}}(0,T;\wi\L^{\frac{r+1}{r}}).
\end{align}
Using the convergences \eqref{3p60}, \eqref{3p62} and \eqref{3p65}, one easily obtains $\G_n(\u_n(t))\to \G(\u(t)),$   for a.e. $t\in[0,T]$ in $\V'+\wi\L^{\frac{r+1}{r}}$, where $\G_n(\u_n)$ is defined in \eqref{eqn-gn}. Note that for $d=3$, $1\leq r\leq 3$ implies $2\leq \frac{r+1}{r}\leq\frac{4}{3}$ and $\V\cap\wi\L^{r+1}\cong\V$. Since $\u\in\mathrm{L}^{\infty}(0,T;\H)\cap\mathrm{L}^2(0,T;\V)$ and $\frac{\d\u}{\d t}\in \mathrm{L}^{\frac{4}{d}}(0,T;\V')$ (see \eqref{3p16}), we immediately have $\u\in\mathrm{W}^{1,\frac{4}{d}}(0,T;\V')$ so that an application of \cite[Theorem 2, pp. 302]{LCE} yields $\u\in\C([0,T];\V')$. Since $\H$ is reflexive, the embedding $\H\hookrightarrow\V'$ is continuous and $\u\in\mathrm{L}^{\infty}(0,T;\H)\cap\C([0,T];\V')$, an application of \cite[Lemma 8.1, Chapter 3, pp. 275]{JLLEM} yields $\u\in\C_w([0,T];\H)$, that is, 
\begin{align}
	\lim_{t\to t_0}(\u(t)-\u(t_0),\v)=0\ \text{ for all }\ \v\in\H. 
\end{align} 

By the definition of $\u_0^n\to\u_0$ in $\H$, $\f_n\to\f$ in $\mathrm{L}^2(0,T;\V')$ and $\u^n\to\u$ in $\mathrm{L}^2(0,T;\H)$, we deduce
\begin{align}
	\lim_{n\to\infty}\left\{2\int_0^t\langle\f_n(s),\u^n(s)\rangle\d s+\|\u_0^n\|_{\H}^2\right\}=2\int_0^t\langle\f(s),\u(s)\rangle\d s+\|\u_0\|_{\H}^2. 
\end{align}
 Since the norm is weakly semi-continuous, $\u^n\xrightarrow{w}\u\in\mathrm{L}^{\infty}(0,T;\H)\cap\mathrm{L}^2(0,T;\V)\cap\mathrm{L}^{r+1}(0,T;\wi\L^{r+1})$ implies 
 \begin{align*}
 &	\liminf_{n\to\infty}\left\{\|\u^n(t)\|_{\H}^2+2\mu\int_0^t\|\u^n(s)\|_{\V}^2\d s+ 2\beta\int_0^t\|\u^n(s)\|_{\widetilde\L^{r+1}}^{r+1}\d s\right\}\nonumber\\&\geq \|\u(t)\|_{\H}^2+2\mu\int_0^t\|\u(s)\|_{\V}^2\d s+ 2\beta\int_0^t\|\u(s)\|_{\widetilde\L^{r+1}}^{r+1}\d s,
 \end{align*}
 and the energy inequality \eqref{ener-ine} follows. On the other hand, energy inequality immediately gives
 \begin{align*}
 	\limsup_{t\to 0}\|\u(t)\|_{\H}^2\leq\|\u_0\|_{\H}^2,
 	\end{align*}
 	whereas $\u\in\C_w([0,T];\H)$ provides 
 	\begin{align*}
 			\liminf_{t\to 0}\|\u(t)\|_{\H}^2\geq\|\u_0\|_{\H}^2,
 	\end{align*}
 	The above two relations imply 
 		\begin{align*}
 		\lim_{t\to 0}\|\u(t)\|_{\H}^2=\|\u_0\|_{\H}^2.
 	\end{align*}
 	This relation together with the $\H$-weak continuity of $\u$ allows us to conclude that \eqref{360} holds. 
 	
 	For $d=2$, the fact that $\u\in\mathrm{L}^{\infty}(0,T;\H)\cap\mathrm{L}^2(0,T;\V)$ and an application of Ladyzhenskaya inequality yield $\u\in\mathrm{L}^4(0,T;\wi\L^4)$. This fact together with \cite[Theorem 4.1, pp. 23]{GGP} imply $\u$  verifies the energy equality \eqref{3.47}, so that $\u\in\C([0,T];\H)$ also. In order to prove the uniqueness of weak solutions, we estimate $\langle\B(\u_1-\u_2,\u_2),\u_1-\u_2\rangle$ from \eqref{346} as
 	\begin{align}\label{3711}
 		|\langle\B(\u_1-\u_2,\u_2),\u_1-\u_2\rangle|&=|\langle\B(\u_1-\u_2,\u_1-\u_2),\u_2\rangle|\nonumber\\&\leq\|\u_2\|_{\wi\L^4}\|\u_1-\u_2\|_{\V}\|\u_1-\u_2\|_{\wi\L^4}\nonumber\\&\leq 2^{1/4}\|\u_2\|_{\wi\L^4}\|\u_1-\u_2\|_{\V}^{3/2}\|\u_1-\u_2\|_{\H}^{1/2}\nonumber\\&\leq\frac{\mu}{2}\|\u_1-\u_2\|_{\V}^2+\frac{27}{32\mu^3}\|\u_2\|_{\wi\L^4}^4\|\u_1-\u_2\|_{\H}^2. 
 	\end{align}
 	Using \eqref{3711} in \eqref{346}, we arrive at 
 	\begin{align}\label{372}
 		&\|\u_1(t)-\u_2(t)\|_{\H}^2+\mu \int_0^t\|\u_1(s)-\u_2(s)\|_{\V}^2\d s+\frac{\beta}{2^{r-1}}\int_0^t\|\u_1(s)-\u_2(s)\|_{\wi\L^{r+1}}^{r+1}\d s\nonumber\\&\leq\|\u_1(0)-\u_2(0)\|_{\H}^2+\frac{27}{16\mu^3}\int_0^t\|\u_2(s)\|_{\wi\L^4}^4\|\u_1(s)-\u_2(s)\|_{\H}^2\d s.
 	\end{align}
 	An application of Gr\"onwall's inequality in \eqref{372} yields 
 	\begin{align}
 		\|\u_1(t)-\u_2(t)\|_{\H}^2\leq \|\u_1(0)-\u_2(0)\|_{\H}^2\exp\left\{\frac{27}{16\mu^3}\int_0^T\|\u_2(t)\|_{\wi\L^4}^4\d t\right\}, 
 	\end{align}
 	and the uniqueness $\u_1(t)=\u_2(t)$ for all $t\in[0,T]$ in $\H$ follows since  $\u\in\mathrm{L}^4(0,T;\wi\L^4)$. 
\end{proof}

\begin{theorem}\label{main3}
	For $d=r=3$, 	let $\u_0\in \H$ and $\f\in\mathrm{L}^{2}(0,T;\V')$  be given.  Then, the weak solution obtained in Theorem \ref{main} is unique. 
\end{theorem}
\begin{proof}
	Let $\u\in\mathrm{L}^{\infty}(0,T;\H)\cap\mathrm{L}^2(0,T;\V)\cap\mathrm{L}^{4}(0,T;\wi\L^4)$ be a weak solution of the problem \eqref{kvf}. 	Let us now consider the following problem: 
	\begin{equation}\label{eqn-ch}
		\left\{
		\begin{aligned}
			\frac{\partial \v}{\partial t}-\mu \Delta\v+(\u\cdot\nabla)\v+\beta|\v|^{2}\v+\nabla q&=\boldsymbol{f}, \ \text{ in } \ \mathcal{O}\times(0,T), \\ \nabla\cdot\v&=0, \ \text{ in } \ \mathcal{O}\times(0,T), \\
			\v&=\mathbf{0}\ \text{ on } \ \partial\mathcal{O}\times(0,T), \\
			\v(0)&=\u_0 \ \text{ in } \ \mathcal{O}. 
		\end{aligned}
		\right.
	\end{equation}
	One taking orthogonal projection $\mathcal{P}$ onto the equation \eqref{eqn-ch}, we obtain  for all $t\in(0,T)$
	\begin{equation}\label{eqn-ch1}
		\left\{
		\begin{aligned}
			\frac{\d\v(t)}{\d t}+\mu \A\v(t)+\B(\u(t),\v(t))+ \beta\mathcal{C}(\v(t))&=\f(t),\\
			\v(0)&=\u_0\in\H,
		\end{aligned}
		\right.
	\end{equation}
	where $\B(\v,\u)=\mathcal{P}[(\v\cdot\nabla)\u]$ and $\mathcal{C}(\v)=\mathcal{P}(|\v|^2\v)$. For simplicity of notation, we used $\f$ as $\mathcal{P}\f$ again. Since $\langle\B(\u,\v),\v\rangle=0$, $\u_0\in\H$ and $\f\in\mathrm{L}^{2}(0,T;\V')$, the existence of a unique weak solution $\v\in\mathrm{L}^{\infty}(0,T;\H)\cap\mathrm{L}^2(0,T;\V)\cap\mathrm{L}^{4}(0,T;\wi\L^4)$ can be proved in a similar way as in the proof of Theorem \ref{main} or \ref{main2}. Moreover, calculations similar to Step (iii) in the proof of Theorem \ref{main} yield $\v(\cdot)$ satisfies the following energy equality: 
	\begin{align}
	\|\v(t)\|_{\H}^2+2\mu\int_0^t\|\v(s)\|_{\V}^2\d s+ 2\beta\int_0^t\|\v(s)\|_{\widetilde\L^{r+1}}^{r+1}\d s=\|\u_0\|_{\H}^2+\int_0^t\langle\f(s),\v(s)\rangle\d s,
	\end{align} 
	for all $t\in[0,T]$. Let us now show that the weak solution of the problem obtained above is unique. For any given data $\u_0\in\H$ and $\f\in\mathrm{L}^2(0,T;\H)$, let us assume that $\v_1$ and $\v_2$ are two weak solutions of the problem \eqref{eqn-ch1}. Then $\v=\v_1-\v_2$ satisfies the following for all $t\in(0,T)$: 
		\begin{equation}\label{eqn-ch2}
		\left\{
		\begin{aligned}
			\frac{\d\v(t)}{\d t}+\mu \A\v(t)+\B(\u(t),\v(t))+ \beta[\mathcal{C}(\v_1(t))-\mathcal{C}(\v_2(t))]&=\mathbf{0},\\
			\v(0)&=\boldsymbol{0}.
		\end{aligned}
		\right.
	\end{equation}
	Taking the inner product with $\v$ to the first equation in \eqref{eqn-ch2} provides 
	\begin{align*}
		\frac{1}{2}\frac{\d}{\d t}\|\v(t)\|_{\H}^2+\mu\|\nabla\v(t)\|_{\H}^2+\beta\langle\mathcal{C}(\v_1(t))-\mathcal{C}(\v_2(t)),\v(t)\rangle =\boldsymbol{0},
	\end{align*}
	since $\langle\B(\u,\v),\v\rangle=0$. Integrating the above inequality from $0$ to $t$, and then using the monotonicity property of the operator  $\mathrm{C}(\cdot)$ (see \eqref{2.23}), we deduce 
	\begin{align}\label{378}
		\|\v(t)\|_{\H}^2+2\mu\int_0^t\|\nabla\v(s)\|_{\H}^2\d s\leq\|\v(0)\|_{\H}^2=\boldsymbol{0},
	\end{align}
	which immediately gives $\v_1(t)=\v_2(t)$ for all $t\in[0,T]$ in $\H$. 
	
	Our next aim is to show that $\v(t,x)=\u(t,x)$ for all $t\in[0,T]$ and a.e. $x\in\Omega$, where $\u$ is a weak solution of \eqref{kvf} corresponding to the initial data $\u_0$ and forcing $\f$, which is appearing in \eqref{eqn-ch1}. Let us define $\w=\u-\v$ and take the difference between \eqref{kvf} and \eqref{eqn-ch1} to find 
		\begin{equation}\label{eqn-ch3}
		\left\{
		\begin{aligned}
			\frac{\d\w(t)}{\d t}+\mu \A\w(t)+\B(\u(t),\w(t))+ \beta[\mathcal{C}(\u(t))-\mathcal{C}(\v(t))]&=\mathbf{0},\\
			\w(0)&=\boldsymbol{0}.
		\end{aligned}
		\right.
	\end{equation}
	The systems \eqref{eqn-ch2} and \eqref{eqn-ch3} are the same and a calculation similar to \eqref{378} yields $\u(t)=\v(t)$ for all $t\in[0,T]$ in $\H$. The uniqueness of the weak solutions of the problem \eqref{eqn-ch1} yields the uniqueness of weak solutions of the problem \eqref{kvf} also. 
\end{proof}

\begin{remark}\label{rem-con}
	The results obtained in Theorem \ref{main} provides continuous dependence on the initial data (Part (2) of the proof), whereas Theorem \ref{main3} provides uniqueness result only. 
\end{remark}


\begin{remark}\label{rem3.7}
		If the  domain is an $d$-dimensional torus, then one can approximate functions in $\L^p$-spaces using the truncated Fourier expansions in the following way (see \cite[Theorem 1.6, Chapter 1, pp. 27]{JCR4} and \cite[Theorem 5.2]{KWH}). Let $\mathcal{O}=[0,2\pi]^d$ and $\mathcal{Q}_k:=[-k,k]^d\cap\mathbb{Z}^d$. For every $\w\in\L^1(\mathcal{O})$ and every $k\in\mathbb{N}$, we define $R_k(\w):=\sum_{m\in\mathcal{Q}_k}\widehat{\w}_m e^{\iota m\cdot x},$ where the Fourier coefficients $\widehat{\w}$ are given by $\widehat{\w}:=\frac{1}{|\mathcal{O}|}\int_{\mathcal{O}}\w(x)e^{-\iota m\cdot x}\d x$. Then, for every $1<p<\infty$, there exists a constant $C_p$, independent of $k$, such that $$\|R_k\w\|_{\L^p(\mathcal{O})}\leq C_p\|\w\|_{\L^p(\mathcal{O})}, \ \text{ for all } \ \w\in\L^p(\mathcal{O}),$$ and $$\|R_k\w-\w\|_{\L^p(\mathcal{O})}\to 0\ \text{ as }\ k\to\infty.$$
\end{remark}

\section{Strong solutions}\label{sec4}\setcounter{equation}{0}
In this section, we show the existence of a global strong solution to the system \eqref{kvf} in periodic domains. 
Due to technical difficulties, we obtain the regularity of strong solutions in periodic domain in different function spaces under different regularity assumptions on the initial data and forcing.

\subsection{Functional setting for periodic domain}
For $\mathrm{L}>0$, let us consider a $d$-dimensional torus $\mathbb{T}^d=\left(\frac{\R}{\mathrm{L}\mathbb{Z}}\right)^d$, $2\leq d\leq4$.
Let $\C_{\mathrm{p}}^{\infty}(\mathbb{T}^d;\R^d)$ denote the space of all infinitely differentiable  functions $\u$ satisfying periodic boundary conditions $\u(x+\mathrm{L}e_{i},\cdot) = \u(x,\cdot)$, for $x\in \R^d$. \emph{We are not assuming the zero mean condition for the velocity field unlike the case of NSE, since the absorption term $\beta|\u|^{r-1}\u$ does not preserve this property (see \cite{MTT}). Therefore, we cannot use the well-known Poincar\'e inequality and we have to deal with the  full $\H^1$-norm.} The Sobolev space  $\H_{\mathrm{p}}^s(\mathbb{T}^d):=\mathrm{H}_{\mathrm{p}}^s(\mathbb{T}^d;\mathbb{R}^d)$ is the completion of $\C_{\mathrm{p}}^{\infty}(\mathbb{T}^d;\R^d)$  with respect to the $\H^s$-norm and the norm on the space $\H_{\mathrm{p}}^s(\mathbb{T}^d)$ is given by $$\|\u\|_{{\H}^s_{\mathrm{p}}}:=\left(\sum_{0\leq|\boldsymbol\alpha|\leq s}\|\D^{\boldsymbol\alpha}\u\|_{\mathbb{L}^2(\mathbb{T}^d)}^2\right)^{1/2}.$$ 	
It is known from \cite{JCR1} that the Sobolev space of periodic functions $\H_{\mathrm{p}}^s(\mathbb{T}^d)$, for $s\geq0$, can be defined as
$$\H_{\mathrm{f}}^s(\mathbb{T}^d)=\left\{\u:\u=\sum_{k\in\mathbb{Z}^d}\u_{k}\mathrm{e}^{2\pi i k\cdot x /  \mathrm{L}},\ \overline{\u}_{k}=\u_{-k}, \  \|\u\|_{{\H}^s_\mathrm{f}}:=\left(\sum_{k\in\mathbb{Z}^d}(1+|k|^{2s})|\u_{k}|^2\right)^{1/2}<\infty\right\}.$$ We infer from \cite[Proposition 5.38]{JCR1} that the norms $\|\cdot\|_{{\H}^s_{\mathrm{p}}}$ and $\|\cdot\|_{{\H}^s_f}$ are equivalent. Let us define 
\begin{align*} 
	\mathcal{V}:=\{\u\in\C_{\mathrm{p}}^{\infty}(\mathbb{T}^d;\R^d):\nabla\cdot\u=0\}.
\end{align*}
We define the spaces $\H$ and $\widetilde{\L}^{p}$ as the closure of $\mathcal{V}$ in the Lebesgue spaces $\mathrm{L}^2(\mathbb{T}^d;\R^d)$ and $\mathrm{L}^p(\mathbb{T}^d;\R^d)$ for $p\in(2,\infty)$, respectively. We also define the space $\V$ as the closure of $\mathcal{V}$ in the Sobolev space $\mathrm{H}^1(\mathbb{T}^d;\R^d)$. Then, we characterize the spaces $\H$, $\widetilde{\L}^p$ and $\V$ with the norms  $$\|\u\|_{\H}^2:=\int_{\mathbb{T}^d}|\u(x)|^2\d x,\quad \|\u\|_{\widetilde{\L}^p}^p:=\int_{\mathbb{T}^d}|\u(x)|^p\d x\ \text{ and }\ \|\u\|_{\V}^2:=\int_{\mathbb{T}^d}(|\u(x)|^2+|\nabla\u(x)|^2)\d x,$$ respectively.  We define the Stokes operator $\A\u:=-\mathcal{P}\Delta\u=-\Delta\u,\;\u\in\D(\A)=\V\cap{\H}^{2}_\mathrm{p}(\mathbb{T}^d). $ From now onward, we set $\mathcal{O}:=\mathbb{T}^d$.

In  periodic domains or on the whole space $\R^d,$ the operators $\mathcal{P}$ and $-\Delta$ commute, so that we can use \eqref{3} and we have the following result  also (see \cite[Lemma 2.1]{KWH}): 
\begin{align}\label{370}
	0&\leq\int_{\mathcal{O}}|\nabla\u(x)|^2|\u(x)|^{r-1}\d x\leq\int_{\mathcal{O}}|\u(x)|^{r-1}\u(x)\cdot\A\u(x)\d x\nonumber\\&\leq r\int_{\mathcal{O}}|\nabla\u(x)|^2|\u(x)|^{r-1}\d x.
\end{align}
Note that the estimate \eqref{370} is true even in bounded domains (with Dirichlet boundary conditions) if one replaces $\A\u$ with $-\Delta\u$ and \eqref{371}-\eqref{3a71} (see below) holds true in bounded domains as well as on the whole space $\R^d$.
Since we are not assuming zero mean condition, using the Sobolev embedding and  \cite[Lemma 2]{JCRWS}, we infer
\begin{align}
	\|\u\|^{p+q}_{\L^{\frac{(p+q)d}{d-p}}(\mathcal{O})}&=\||\u|^{\frac{p+q}{p}}\|_{\L^{\frac{pd}{p-d}}(\mathcal{O})}^p\leq C\||\u|^{\frac{p+q}{p}}\|_{\mathrm{W}^{1,p}(\mathcal{O})}^p\\&=C\left(\|\nabla(|\u|^{\frac{p+q}{p}})\|_{\L^p(\mathcal{O})}^p+\|\u\|_{\L^{p+q}(\mathcal{O})}^{p+q}\right)\nonumber\\&\leq C\left(\||\u|^{\frac{q}{p}}|\nabla\u|\|_{\L^p(\mathcal{O})}^p+\|\u\|_{\L^{p+q}(\mathcal{O})}^{p+q}\right)\nonumber\\&= C\bigg(\int_{\mathcal{O}}|\nabla\u(x)|^{p}|\u(x)|^q\d x+
	\int_{\mathcal{O}}|\u(x)|^{p+q}\d x\bigg),
\end{align}
for all $\u\in\W^{1,m}_0(\mathcal{O})$  with $m=d(p+q)/(d+q),p<d$. One can handle the case $d=2$ in the following way: Let us take $\u\in\C_{\mathrm{p}}^{\infty}(\mathcal{O})$. Then $|\u|^{p/2}\in\H_{\mathrm{p}}^1(\mathcal{O})$, for all $p\in[2,\infty)$, and from the Sobolev embedding, $\H_{\mathrm{p}}^1(\mathcal{O})\hookrightarrow\L^p(\mathcal{O})$, for all $p\in[2,\infty)$, we find 
\begin{align*}
	\|\u\|^{r+1}_{\L^{p(r+1)}(\mathcal{O})}&=\||\u|^{\frac{r+1}{2}}\|_{\L^{2p}(\mathcal{O})}^2\leq C\||\u|^{\frac{r+1}{2}}\|_{\H_{\mathrm{p}}^1(\mathcal{O})}^2\nonumber\\&\leq C\bigg(\int_{\mathcal{O}}|\nabla|\u(x)|^{\frac{r+1}{2}}|^2\d x+
	\int_{\mathcal{O}}||\u(x)|^{\frac{r+1}{2}}|^2\d x\bigg).
\end{align*}
Buy using elementary calculus identity $\partial_k(|\u|^{r+1})=(r+1)u_k(\partial u_k)|\u|^{r-1}$, where $\partial_k$ denotes the $k$-th partial derivative, we infer that $|\nabla|\u|^{\frac{r+1}{2}}|^2\leq C_r|\u|^{r-1}|\nabla\u|^2$ (see the proof of \cite[Lemma 1]{JCRWS}). Thus, we further have 
\begin{align}\label{3a71}
	\|\u\|^{r+1}_{\L^{p(r+1)}(\mathcal{O})}\leq C\bigg(\int_{\mathcal{O}}|\nabla\u(x)|^2|\u(x)|^{r-1}\d x+ \int_{\mathcal{O}}|\u(x)|^{r+1}\d x\bigg),
\end{align}
for any $p\in[2,\infty)$ and $r\geq1$. Similarly, for $d=3$, by the Sobolev embedding $\H_{\mathrm{p}}^1(\mathcal{O})\hookrightarrow\L^6(\mathcal{O})$, we find
\begin{align}\label{371}
	\|\u\|^{r+1}_{\L^{3(r+1)}(\mathcal{O})}&=\||\u|^{\frac{r+1}{2}}\|_{\L^{6}(\mathcal{O})}^2\leq C\||\u|^{\frac{r+1}{2}}\|_{\H_{\mathrm{p}}^1(\mathcal{O})}^2
	\nonumber\\&\leq C\bigg(\int_{\mathcal{O}}|\nabla\u(x)|^2|\u(x)|^{r-1}\d x+ \int_{\mathcal{O}}|\u(x)|^{r+1}\d x\bigg), 
\end{align}
for all $\u\in\D(\A)$ and for $r\geq1$. For $d=4$, we obtain 
\begin{align}\label{3p71}
	\|\u\|_{\widetilde{\L}^{2(r+1)}}^{r+1}=\||\u|^{\frac{r+1}{2}}\|_{\L^{4}(\mathcal{O})}^2\leq C\bigg(\int_{\mathcal{O}}|\nabla\u(x)|^2|\u(x)|^{r-1}\d x+\int_{\mathcal{O}}|\u(x)|^{r+1}\d x\bigg), 
\end{align}
for all $\u\in\D(\A)$ and $r\geq1$.

 In the definition given below, we fix $p\in[2,\infty)$ for $d=2$, $p=3$ for $d=3$ and $p=2$ for $d=4$.

  \begin{definition}
	We say that $\u(\cdot)$ is a \emph{strong solution} of the system \eqref{kvf} for $r>3$ ($2\leq d\leq 4$) and $r=3$  ($d=3,4$) with $2\beta\mu\geq1$,  corresponding to the initial data $\u_0\in\V$ and forcing $\f\in\mathrm{L}^2(0,T;\H)$ if it satisfies the following:
	\begin{enumerate}
		\item[1.)]  $\u$ is weak solution of \eqref{kvf},
		\item[2.)]  $\u\in\mathrm{L}^{\infty}(0,T;\V)\cap\mathrm{L}^2(0,T;\D(\A))\cap\mathrm{L}^{r+1}(0,T;\widetilde\L^{p(r+1)})$,
		\item[3.)]  \eqref{kvf} holds for a.e. $t\in[0,T]$ in $\H$. 
	\end{enumerate}
\end{definition}
 
The first part of the following result is available in the literature (for instance, see \cite[Theorem 3.1]{KWH}). 
\begin{theorem}[Regualrity]\label{reg}
For $\u_0\in\V$ and $\f\in\mathrm{L}^2(0,T;\H)$, any weak solution $\u(\cdot)$ to the system \eqref{kvf} with $r>3$, obtained in Theorem \ref{main}  satisfies the following regularity: $$\u\in\mathrm{L}^{\infty}(0,T;\V)\cap\mathrm{L}^2(0,T;\D(\A))\cap\mathrm{L}^{r+1}(0,T;\widetilde\L^{p(r+1)}),$$ and hence $\u(\cdot)$ is a \emph{strong solution} to the system \eqref{kvf}. 
For $r=3$, $\u(\cdot)$ is a \emph{strong solution} to the system \eqref{kvf} with $\u\in\mathrm{L}^{\infty}(0,T;\V)\cap\mathrm{L}^2(0,T;\D(\A))\cap\mathrm{L}^{4}(0,T;\widetilde\L^{4p}),$ provided $2\beta\mu \geq 1$.

Furthermore, if $\u_0\in\V\cap\widetilde\L^{r+1}$, then $$\u\in\C([0,T];\V)\cap\mathrm{L}^{\infty}(0,T;\widetilde\L^{r+1})\cap\mathrm{L}^2(0,T;\D(\A))\cap\mathrm{L}^{r+1}(0,T;\widetilde\L^{p(r+1)}).$$
\end{theorem}

\begin{proof}
	One can refer \cite[Theorem 3.1]{KWH} for a proof of the first part of the theorem. For completeness, we are providing a proof here.

		Taking inner product with $\A\u(\cdot)$ to the first equation in \eqref{kvf}, we find 
	\begin{align}\label{273}
	\frac{1}{2}\frac{\d}{\d t}\|\nabla\u(t)\|_{\H}^2+\mu \|\A\u(t)\|_{\H}^2&=-(\B(\u(t)),\A\u(t))-\beta(\mathcal{C}(\u(t)),\A\u(t))-(\f(t),\A\u(t)),
	\end{align}
	for a.e. $t\in[0,T]$. We estimate the first term on the right hand side of the equality \eqref{273} using H\"older's, and Young's inequalities as 
	\begin{align}\label{275}
	|(\B(\u),\A\u)|&\leq\||\u||\nabla\u|\|_{\H}\|\A\u\|_{\H}\leq\frac{\mu }{4}\|\A\u\|_{\H}^2+\frac{1}{\mu }\||\u||\nabla\u|\|_{\H}^2. 
	\end{align}
	For $r>3$, we  estimate the final term from \eqref{275} using H\"older's and Young's inequalities as 
	\begin{align}\label{rg3c}
&	\int_{\mathcal{O}}|\u(x)|^2|\nabla\u(x)|^2\d x\nonumber\\&=\int_{\mathcal{O}}|\u(x)|^2|\nabla\u(x)|^{\frac{4}{r-1}}|\nabla\u(x)|^{\frac{2(r-3)}{r-1}}\d x\nonumber\\&\leq\left(\int_{\mathcal{O}}|\u(x)|^{r-1}|\nabla\u(x)|^2\d x\right)^{\frac{2}{r-1}}\left(\int_{\mathcal{O}}|\nabla\u(x)|^2\d x\right)^{\frac{r-3}{r-1}}\nonumber\\&\leq\frac{\beta\mu}{2}\left(\int_{\mathcal{O}}|\u(x)|^{r-1}|\nabla\u(x)|^2\d x\right)+ \frac{r-3}{r-1}\left(\frac{4}{\beta\mu(r-1)}\right)^{\frac{2}{r-3}}\left(\int_{\mathcal{O}}|\nabla\u(x)|^2\d x\right).
	\end{align}
	Using Cauchy-Schwarz and Young's inequalities, we estimate $|(\f,\A\u)|$ as 
	\begin{align}\label{276}
	|(\f,\A\u)|\leq\|\f\|_{\H}\|\A\u\|_{\H}\leq\frac{\mu }{4}\|\A\u\|_{\H}^2+\frac{1}{\mu }\|\f\|_{\H}^2.
	\end{align}
	Substituting \eqref{3}, \eqref{275}-\eqref{276} in \eqref{273}, and then integrating from $0$ to $t$, we obtain 
	\begin{align}\label{277}
	&	\|\nabla\u(t)\|_{\H}^2+\mu \int_0^t\|\A\u(s)\|_{\H}^2\d s+\beta\int_0^t\||\u(s)|^{\frac{r-1}{2}}|\nabla\u(s)|\|_{\H}^2\d s\nonumber\\&\leq\|\nabla\u_0\|_{\H}^2+\frac{2}{\mu }\int_0^t\|\f(s)\|_{\H}^2\d s+\varrho^*\int_0^t\|\nabla\u(s)\|_{\H}^2\d s,
	\end{align}
	for all $t\in[0,T]$, where $\varrho^*=\frac{2(r-3)}{\mu (r-1)}\left(\frac{4}{\beta\mu (r-1)}\right)^{\frac{2}{r-3}}$. An application of Gr\"onwall's inequality (see Lemma \ref{Gronw}) gives 
		\begin{align}\label{2pp77}
	&	\|\nabla\u(t)\|_{\H}^2\leq\left\{\|\nabla\u_0\|_{\H}^2+\frac{2}{\mu }\int_0^T\|\f(t)\|_{\H}^2\d t\right\} e^{\varrho^*T},
	\end{align}
for all $t\in[0,T]$. Thus,	it is immediate that  
		\begin{align}\label{2p77}
	&\sup_{t\in[0,T]}	\|\nabla\u(t)\|_{\H}^2+\mu \int_0^T\|\A\u(t)\|_{\H}^2\d t+\beta\int_0^T\||\u(t)|^{\frac{r-1}{2}}|\nabla\u(t)|\|_{\H}^2\d t\nonumber\\&\leq \left\{\|\nabla\u_0\|_{\H}^2+\frac{2}{\mu }\int_0^T\|\f(t)\|_{\H}^2\d t\right\}e^{\varrho^*T}.
	\end{align}
	The fact that $\u\in \mathrm{L}^{r+1}(0,T;\widetilde\L^{p(r+1)})$ can be obtained from \eqref{3a71}-\eqref{3p71}.

	For $r=3$, we estimate $|(\B(\u),\A\u)|$ as 
	\begin{align}\label{291}
	|(\B(\u),\A\u)|&\leq\|(\u\cdot\nabla)\u\|_{\H}\|\A\u\|_{\H}\leq\frac{1}{4\theta}\|\A\u\|_{\H}^2+\theta\||\u||\nabla\u|\|_{\H}^2,
	\end{align}
where $\theta\in(0,1)$. A calculation similar to \eqref{277} gives 
	\begin{align}\label{292}
		&	\|\u(t)\|_{\V}^2+\left(\mu -\frac{1}{2\theta}\right)\int_0^t\|\A\u(s)\|_{\H}^2\d s+2\left(\beta-\theta\right)\int_0^t\||\u(s)||\nabla\u(s)|\|_{\H}^2\d s\nonumber\\&\leq\|\u_0\|_{\V}^2+\frac{1}{\mu }\int_0^t\|\f(s)\|_{\H}^2\d s.
	\end{align} 
	For $2\beta\mu \geq 1$,  it is immediate that 
	\begin{align}\label{2.94}
	&\sup_{t\in[0,T]}	\|\nabla\u(t)\|_{\H}^2+\left(\mu -\frac{1}{2\theta}\right)\int_0^T\|\A\u(t)\|_{\H}^2\d t+\left(\beta-\theta\right)\int_0^T\||\u(t)||\nabla\u(t)|\|_{\H}^2\d t\nonumber\\&\leq \left\{\|\nabla\u_0\|_{\H}^2+\frac{2}{\mu }\int_0^T\|\f(t)\|_{\H}^2\d t\right\}.
	\end{align}

Let us  take $\u_0\in\V\cap\widetilde\L^{r+1}$ and prove that $\u\in\C([0,T];\V)\cap\mathrm{L}^{\infty}(0,T;\widetilde\L^{r+1})$.	Taking the inner product with $\frac{\d\u}{\d t}$ to the first equation in \eqref{kvf}, we obtain 
	\begin{align}\label{295}
	\left\|\frac{\d\u}{\d t}\right\|_{\H}^2+\frac{\mu}{2}\frac{\d}{\d t} \|\nabla\u(t)\|_{\H}^2+\left(\B(\u(t)),\frac{\d\u}{\d t}\right)+ \beta\left(\mathcal{C}(\u(t)),\frac{\d\u}{\d t}\right)= \left(\f(t),\frac{\d\u}{\d t}\right).
	\end{align}
	It can be seen that 
	\begin{align*}
	\left(\mathcal{C}(\u(t)),\frac{\d\u}{\d t}\right) = \left(\mathcal{P}(|\u(t)|^{r-1}\u(t)),\frac{\d\u}{\d t}\right)= \frac{1}{r+1}\frac{\d}{\d t}\|\u(t)\|_{\widetilde\L^{r+1}}^{r+1}. 
	\end{align*}
	We estimate the terms $|(\B(\u),\frac{\d\u}{\d t}|$ and $|(\f,\frac{\d\u}{\d t})|$ using H\"older's and Young's inequalities as
	\begin{align*}
	\left|\left(\B(\u),\frac{\d\u}{\d t}\right)\right| &\leq\|(\u\cdot\nabla)\u\|_{\H}\left\|\frac{\d\u}{\d t}\right\|_{\H}\leq \frac{1}{4}\left\|\frac{\d\u}{\d t}\right\|_{\H}^2+\||\u||\nabla\u|\|_{\H}^2, \\
	\left|\left(\f,\frac{\d\u}{\d t}\right)\right|&\leq \|\f\|_{\H}\left\|\frac{\d\u}{\d t}\right\|_{\H}\leq \frac{1}{4}\left\|\frac{\d\u}{\d t}\right\|_{\H}^2+\|\f\|_{\H}^2. 
	\end{align*}
	From \eqref{295}, we have 
	\begin{align}\label{296}
&	\frac{1}{2}\left\|\frac{\d\u}{\d t}(t)\right\|_{\H}^2+ \frac{\beta}{r+1}\frac{\d}{\d t}\|\u(t)\|_{\widetilde\L^{r+1}}^{r+1}+\frac{\mu }{2}\frac{\d}{\d t}\|\nabla\u(t)\|_{\H}^2\leq\|\f(t)\|_{\H}^2+\||\u||\nabla\u|\|_{\H}^2.
	\end{align}
	Integrating the above inequality from $0$ to $t$, we further have 
	\begin{align}\label{297}
&\frac{\beta}{r+1}	\|\u(t)\|_{\widetilde\L^{r+1}}^{r+1}+\frac{\mu }{2}\|\nabla\u(t)\|_{\H}^2+\frac{1}{2}\int_0^t
\left\|\frac{\d\u}{\d t}(s)\right\|_{\H}^2\d s\nonumber\\&\leq \frac{\beta}{r+1}\|\u_0\|_{\widetilde\L^{r+1}}^{r+1}+\frac{\mu }{2}\|\nabla\u_0\|_{\H}^2+\int_0^t\|\f(s)\|_{\H}^2\d s+\int_0^t\||\u(s)||\nabla\u(s)|\|_{\H}^2\d s\nonumber\\&\leq C\left(\mu,\beta,r,T,\|\u_0\|_{\V},\|\u_0\|_{\widetilde\L^{r+1}},\|\f\|_{\mathrm{L}^2(0,T;\H)}\right),
	\end{align}
	for all $t\in[0,T]$. Note the final term on the right hand side of \eqref{297} is bounded for all $r>3$ (see \eqref{2p77}) and for $r=3$ with $2\beta\mu\geq1$ (see \eqref{2.94}). Thus, we get $\frac{\d\u}{\d t}\in\mathrm{L}^2(0,T;\H)$ along with $\u\in\mathrm{L}^2(0,T;\D(\A))$  and hence by an application of the Lions-Magenes lemma (\cite[Lemma 1.2, pp. 176–177]{Te1}) that $\u\in\C([0,T];\V)$. Moreover, from \eqref{297}, we have $\u\in\mathrm{L}^{\infty}(0,T;\widetilde\L^{r+1})$. 
\end{proof}

\begin{remark}
It can also be seen that 
\begin{align}
\left|\left(\frac{\d\u}{\d t},\v\right)\right|&\leq\mu  |(\A\u,\v)|+|(\B(\u),\v)|+\beta|(\mathcal{C}(\u),\v)|+|(\f,\v)|\nonumber\\&\leq\mu\|\A\u\|_{\H}\|\v\|_{\H}+\||\u||\nabla\u|\|_{\H}\|\v\|_{\H}+\beta\|\u\|_{\wi\L^{2r}}^{r}\|\v\|_{\H},
\end{align}
for any $\v\in\H$. Now, for all $\v\in\mathrm{L}^{r+1}(0,T;\H)$, we find 
\begin{align*}
\int_0^T\left|\left(\frac{\d\u(t)}{\d t},\v(t)\right)\right|\d t&\leq \left\{T^{\frac{r-1}{2(r+1)}}\left[\left(\int_0^T\|\A\u(t)\|_{\H}^2\d t\right)^{1/2}+\left(\int_0^T\||\u(t)||\nabla\u(t)|\|_{\H}^2\d t\right)^{1/2}\right]\right.\nonumber\\&\quad\left.+\beta\left(\int_0^T\|\u(t)\|_{\wi\L^{2r}}^{r+1}\d t\right)^{\frac{r}{r+1}}\right\}\left(\int_0^T\|\v(t)\|_{\H}^{r+1}\d t\right)^{\frac{1}{r+1}}.
\end{align*}
Thus, we obtain $\frac{\d\u}{\d t}\in\mathrm{L}^{\frac{r+1}{r}}(0,T;\H)$ and by Aubin-Lions lemma, we obtain  strongly convergent subsequence of the Galerkin approximated sequence  in $\mathrm{L}^2(0,T;\V)$ and $\C([0,T];\H)$. Note that  for $\u_0\in\V$, we have  $\u\in\C_w([0,T];\V)$ and the map $t\mapsto\|\u(t)\|_{\V}$ is bounded, where $\C_w([0,T];\V)$ denotes the space of functions $\u:[0,T]\to\V$ which are weakly continuous (see \cite[Lemma 8.1, Chapter 3, pp. 275]{JLLEM}). For $3\leq r\leq\frac{2d}{d-2}-1$, we know that $\V\hookrightarrow\wi\L^{r+1}$ and from the second part of Theorem \ref{reg}, we immediately have $\u\in\C([0,T];\V)$. 
\end{remark}


Let us now discuss the existence and uniqueness of strong solutions \eqref{kvf} in the case of periodic domains in a different way. 	We make use of the monotonicity and  demicontinuity properties of the linear and nonlinear operators and an abstract theory developed in \cite{VB1,VB} to get the following result and it holds true for $2\leq d\leq 4$. 
	
	\begin{theorem}\label{thm4.2}
		Let $\f\in\W^{1,1}([0,T];\H)$ and $\u_0\in\V$ be such that $\A\u_0\in\H$. Then, for $r\geq 3$ ($2\beta\mu\geq 1$ for $r=3$), there exists one and only one function $\u:[0,T]\to\V\cap\widetilde\L^{r+1}$ that satisfies: 
		$$\u\in\W^{1,\infty}([0,T];\H),\ \A\u\in\mathrm{L}^{\infty}(0,T;\H),$$ 
		and 
		\begin{equation}\label{holdae}
	\left\{	\begin{aligned}
		\frac{\d\u(t)}{\d t}+\mu\A\u(t)+\B(\u(t))+\beta\mathcal{C}(\u(t))&=\f(t), \ \text{ a.e. }\ t\in[0,T],\\
		\u(0)&=\u_0.
		\end{aligned}
		\right.
		\end{equation}
		Furthermore, $\u(\cdot)$ is everywhere differentiable from the right in $\H$ and 
		\begin{equation}\label{holdfa}
		\frac{\d^+\u(t)}{\d t}+\mu\A\u(t)+\B(\u(t))+\beta\mathcal{C}(\u(t))=\f(t), \ \text{ for all }\ t\in[0,T).
		\end{equation}
	\end{theorem}
\begin{proof}
 We first note that $\mathcal{V}\hookrightarrow\V\cap\widetilde{\L}^{r+1}\hookrightarrow\H$ and $\mathcal{V}$ is dense in $\H,\V$ and $\widetilde{\L}^{r+1},$ and hence $\V\cap\widetilde{\L}^{r+1}$ is dense in $\H$. We have the following continuous  embedding also:
$$\V\cap\widetilde{\L}^{r+1}\hookrightarrow\H\equiv\H'\hookrightarrow\V'+\widetilde\L^{\frac{r+1}{r}}.$$ We define equivalent norms on $\V\cap\widetilde\L^{r+1}$ and $\V'+\widetilde\L^{\frac{r+1}{r}}$ as  (see \cite{NAEG})
\begin{align}
\|\u\|_{\V\cap\widetilde\L^{r+1}}=\left(\|\u\|_{\V}^2+\|\u\|_{\widetilde\L^{r+1}}^2\right)^{1/2}\ \text{ and } \ \|\u\|_{\V'+\widetilde\L^{\frac{r+1}{r}}}=\inf_{\u=\v+\w}\left(\|\v\|_{\V'}^2+\|\w\|_{\widetilde\L^{\frac{r+1}{r}}}^2\right)^{1/2}.
\end{align}
With the above norms, using \cite[Theorem 1]{ICAM}, the space $\V\cap\widetilde\L^{r+1}$ becomes uniformly convex and hence $\V'+\widetilde\L^{\frac{r+1}{r}}=(\V\cap\widetilde{\L}^{r+1})'$ is uniformly smooth (\cite[Theorem 2.13, Chapter II, pp. 52]{ICI}). The Milman-Pettis Theorem states that a uniformly convex Banach space is reflexive (see \cite[Theorem 2.9, Chapter II, pp. 49]{ICI}).  Thus,  it is immediate that $(\V\cap\widetilde{\L}^{r+1})''=\V\cap\widetilde{\L}^{r+1}$ and the space $\V\cap\widetilde{\L}^{r+1}$ is a reflexive Banach space. From \cite[Theorem 1.1, Chapter 1, pp. 2]{VB}, we know that any reflexive Banach space $\X$ can be renormed such that $\X$ and $\X'$ become strictly convex. Furthermore, from \cite[Theorem 1.1, Chapter 1, pp. 2]{VB} (see \cite[Theorem 2.9, Chapter III, pp. 108]{ICI}), we infer that if the dual space $\X'$ is strictly convex, then the duality mapping $\mathcal{J}:\X\to\X'$ is single valued and demicontinuous (see \cite[Corollary 1.5, Chapter II, pp. 43]{ICI} also).  With the above equivalent norms, the spaces $\V\cap\widetilde{\L}^{r+1}$ and $\V'+\widetilde\L^{\frac{r+1}{r}}$ become strictly convex and  hence the duality mapping $\mathcal{J}:\V\cap\widetilde{\L}^{r+1}\to\V'+\widetilde\L^{\frac{r+1}{r}}$ defined by
\begin{align*}
\langle\mathcal{J}(\u),\v\rangle =\int_{\mathcal{O}}\nabla\u(x)\cdot\nabla\v(x)\d x+ \|\u\|_{\widetilde\L^{r+1}}^{1-r}\int_{\mathcal{O}}|\u(x)|^{r-1}\u(x)\cdot\v(x)
\d x,
\end{align*}
for all $\u,\v\in\V\cap\widetilde\L^{r+1}$, is single valued. Note that $\mathcal{J}(\u)=-\A\u+\|\u\|_{\widetilde\L^{r+1}}^{1-r}|\u|^{r-1}\u,$ and by using \eqref{22}, we have 
   \begin{align*}
	\|\mathcal{J}(\u)\|_{\V'+\widetilde\L^{\frac{r+1}{r}}}^2&=\inf_{\mathcal{J}(\u)=-\A\u+\|\u\|_{\widetilde\L^{r+1}}^{1-r}|\u|^{r-1}\u}\left(\|\A\u\|_{\V'}^2+\|\u\|_{\widetilde\L^{r+1}}^{2(1-r)}\||\u|^{r-1}\u\|_{\widetilde\L^{\frac{r+1}{r}}}^2\right)\nonumber\\&= \|\u\|_{\V}^2+\|\u\|_{\widetilde\L^{r+1}}^2,
  \end{align*} 
	since $\mathcal{J}(\u)$ is single valued.  Thus, we get $$\langle\mathcal{J}(\u),\u\rangle=\|\u\|_{\V}^2+\|\u\|_{\widetilde\L^{r+1}}^2=\|\u\|_{\V\cap\widetilde\L^{r+1}}^2=\|\mathcal{J}(\u)\|_{\V'+\widetilde\L^{\frac{r+1}{r}}}^2.$$ Thus, one can apply the abstract theory discussed in \cite{VB1,VB}  for our model. Let us now define the operator $\mathscr{A}_{\H}:\H\to\H$ by 
	\begin{align}
		\mathscr{A}_{\H}(\u):=\mu\A\u+\B(\u)+\beta\mathcal{C}(\u), \ \text{ for all } \u\in\D(\mathscr{A}_{\H})=\{\u\in\V\cap\widetilde\L^{r+1}:\A\u\in\H\}.
	\end{align}
From the Theorem \ref{thm2.2}, we know that the operator $\kappa \I_d+\mathscr{A}_{\H}(\cdot)$ is a monotone operator, where $\varrho$ is defined in \eqref{215}. Using Lemma \ref{lem2.8}, it can be easily seen that the operator $\kappa \I_d +\mathscr{A}_{\H}(\cdot): \V\cap\widetilde\L^{r+1}\to \V'+\widetilde\L^{\frac{r+1}{r}}$ is demicontinuous. Let us now establish the coercivity of the operator $\kappa \I_d+\mathscr{A}_{\H}(\cdot)$. 

\vskip 2mm
\noindent
\textbf{Step I:} \textsl{The operator $\kappa \I_d+\mathscr{A}_{\H}(\cdot)$ is coercive.} We consider 
	\begin{align*} 
	\frac{\langle\kappa\u+\mathscr{A}_{\H}(\u),\u\rangle}{\|\u\|_{\V\cap\widetilde\L^{r+1}}}&=\frac{\kappa\|\u\|_{\H}^2+\mu\|\nabla\u\|_{\H}^2+\beta\|\u\|_{\widetilde\L^{r+1}}^{r+1}}{\sqrt{\|\u\|_{\V}^2+\|\u\|_{\widetilde\L^{r+1}}^{2}}}\nonumber\\&\geq\frac{\min\{\kappa,\mu,\beta\}\left(\|\u\|_{\V}^2+\|\u\|_{\widetilde\L^{r+1}}^2-1\right)}{\sqrt{\|\u\|_{\V}^2+\|\u\|_{\widetilde\L^{r+1}}^{2}}},
	\end{align*}
	where we have used the fact that $x^2\leq x^{r+1}+1$, for $x\geq 0$ and $r\geq 1$. Thus, we have 
	\begin{align}
	\lim_{\|\u\|_{\V\cap\widetilde\L^{r+1}}\to\infty}	\frac{\langle\kappa\u+\mathscr{A}_{\H}(\u),\u\rangle}{\|\u\|_{\V\cap\widetilde\L^{r+1}}}=\infty,
	\end{align}
	and hence the operator $\kappa \I_d+\mathscr{A}_{\H}(\cdot)$ is coercive.
	
	\vskip 2mm
	\noindent
	\textbf{Step IV:} \textsl{The operator $\kappa \I_d+\mathscr{A}_{\H}(\cdot)$ $m$-accretive in $\H\times\H$}. Let us define an operator 
	$$\mathcal{F}(\u):=\mu\A\u+\B(\u)+\beta\mathcal{C}(\u)+\kappa\u,$$ where $\D(\mathcal{F})=\{\u\in\V\cap\wi\L^{r+1}:\mu\A\u+\B(\u)+\beta\mathcal{C}(\u)\in\H\}.$ Note that the space $\V\cap\widetilde{\L}^{r+1}$ is reflexive. Since $\mathcal{F}(\cdot)$ is monotone, hemicontinuous and coercive from $\V\cap\widetilde{\L}^{r+1}$ to $\V'+\widetilde{\L}^{\frac{r+1}{r}}$, then by an application of \cite[Example 2.3.7, pp. 26]{OPHB}, we obtain that  $\mathcal{F}$ is maximal monotone in $\H$ with domain $\mathrm{D}(\mathcal{F})\supseteq\mathrm{D}(\mathrm{A}).$ In fact, we shall prove that $\mathcal{F}$ is $m$-accretive for $\kappa$ sufficiently large with $\mathrm{D}(\mathcal{F})=\mathrm{D}(\mathrm{A}).$ Let us consider the operators for some $\delta_1, \delta_2\in(0,1)$ as
	\begin{align}
		\mathcal{F}^{1}(\cdot) &= \mu(1-\delta_1)\mathrm{A}+\beta(1-\delta_2)\mathcal{C}(\cdot),\label{3.3.2}\\
		\mathcal{F}^{2}(\cdot) &= \mu\delta_1\mathrm{A}+\B(\cdot)+\beta\delta_2\mathcal{C}(\cdot)+\kappa\mathrm{I},\label{3.3.33}
	\end{align}
	where $\D(\mathcal{F}^{1})=\{\u\in\V\cap\wi\L^{r+1}:\mathcal{F}^1(\cdot)\in\H\}$ and $\D(\mathcal{F}^{2}) = \{\u\in\V\cap\wi\L^{r+1}:\mathcal{F}^{2}(\cdot)\in\H\}.$
	Taking the inner product with $\u$ in \eqref{3.3.2}, we obtain
	\begin{align*}
		\mu(1-\delta_1)\|\nabla\u\|_{\H}^2+\beta(1-\delta_2)\|\u\|_{\wi\L^{r+1}}^{r+1}\leq(\mathcal{F}^{1}(\u),\u)\leq \|\mathcal{F}^{1}(\u)\|_{\H}\|\u\|_{\H},
	\end{align*}
	so that 
	\begin{align}\label{3.3.5}
		\|\nabla\u\|_{\H}^2\leq\frac{1}{\mu(1-\delta_1)}\|\mathcal{F}^{1}(\u)\|_{\H}\|\u\|_{\H}.
	\end{align}
	Taking the inner product with $\A\u$ in \eqref{3.3.2} and using equality \eqref{3}, we get
	\begin{align*}
		\mu(1-\delta_1)\|\A\u\|_{\H}^{2}+\beta(1-\delta_2) \left[\||\u|^{\frac{r-1}{2}}\nabla\u\|_{\H}^2+\frac{4(r-1)}{(r+1)^2}\||\nabla|\u|^{\frac{r+1}{2}}|\|_{\H}^2\right]=(\mathcal{F}^{1}(\u),\A\u). 
	\end{align*}
	Therefore, we have 
	\begin{align}\label{3.147}
		\|\A\u\|_{\H}\leq\frac{1}{\mu(1-\delta_1)}\|\mathcal{F}^{1}(\u)\|_{\H}\ \text{ which implies }\ \D(\mathcal{F}^{1})\subseteq\D(\A). 
	\end{align}
	Moreover, by using Sobolev embedding $\H^2_{\mathrm{p}}(\mathcal{O})\hookrightarrow\wi\L^{2r}$, we infer  
	\begin{align*}
		\|\mathcal{F}^{1}(\u)\|_{\H}&\leq\mu(1-\delta_1)\|\mathrm{A}\u\|_{\H} +C\beta(1-\delta_2)\|\u\|_{\wi\L^{2r}}^r\nonumber\\&\leq\mu(1-\delta_1)\|\mathrm{A}\u\|_{\H}+C \beta(1-\delta_2)\|\u\|_{\H^2_{\mathrm{p}}(\mathbb{T}^d)}^r\nonumber\\&\leq\mu(1-\delta_1) \|\mathrm{A}\u\|_{\H}+C\beta(1-\delta_2)\|\A\u\|_{\H}^r+C\beta(1-\delta_2)\|\u\|_{\H}^r, 
	\end{align*}
	which gives $\D(\mathcal{F}^{1})\supseteq\D(\A)$ and therefore $\D(\A)=\D(\mathcal{F}^1).$
	Similarly, taking the inner product with $\mathcal{C}(\u)$ in \eqref{3.3.2}, we find
	\begin{align}\label{3.149}
		\|\mathcal{C}(\u)\|_{\H}\leq\frac{1}{\beta(1-\delta_2)}\|\mathcal{F}^{1}(\u)\|_{\H}.
	\end{align}
	For $r>3$, similar to \eqref{rg3c}, we estimate $\|\B(\u)\|_{\H}$ using H\"older's inequality as follows:
	\begin{align}\label{3.3.4}
		\|\B(\u)\|_{\H}^2\leq
		\||\u|^\frac{r-1}{2}\nabla\u\|_{\H}^{\frac{4}{r-1}}\|\nabla\u\|_{\H}^\frac{2(r-3)}{r-1}.  
	\end{align}
	Note that $(\mathcal{C}(\u),\A\u)=\int_{\mathbb{T}^{d}}(-\Delta \u(x))\cdot|\u(x)|^{r-1}\u(x)\d x$. Using the estimate \eqref{3.3.5} and the equality \eqref{3} in \eqref{3.3.4}, we find
	\begin{align*}
		\|\B(\u)\|_{\H}^2\leq\left[(\mathcal{C}(\u),\A\u)\right]^\frac{2}{r-1}\left[\frac{1}{\mu(1-\delta_1)}\|\mathcal{F}^{1}(\u)\|_{\H}\|\u\|_{\H}\right]^\frac{r-3}{r-1}.
	\end{align*}
	Therefore, we estimate $\|\B(\u)\|_{\H}$ as 
	\begin{align}\label{3p36}
		\|\B(\u)\|_{\H}\leq\|\mathcal{C}(\u)\|_{\H}^\frac{1}{r-1}\|\A\u\|_{\H}^\frac{1}{r-1}\left[\frac{1}{\mu(1-\delta_1)}\|\mathcal{F}^{1}(\u)\|_{\H}\|\u\|_{\H}\right]^\frac{r-3}{2(r-1)}.
	\end{align}
	Using the estimates \eqref{3.147}-\eqref{3.149} in \eqref{3p36},  then using Young's inequality, we get
	\begin{align}\label{3.3.6}
		\|\B(\u)\|_{\H}&\leq\left[\frac{\|\mathcal{F}^1(\u)\|_{\H}^2}{\beta\mu(1-\delta_1)(1-\delta_2)}\right]^\frac{1}{r-1}\left[\frac{\|\mathcal{F}^{1}(\u)\|_{\H}\|\u\|_{\H}}{\mu(1-\delta_1)}\right]^\frac{r-3}{2(r-1)}\nonumber\\&=
		\frac{1}{\sqrt{\mu(1-\delta_1)}}\left[\frac{1}{\beta(1-\delta_2)}\right]^\frac{1}{r-1}\|\mathcal{F}^1(\u)\|_{\H}^\frac{r+1}{2(r-1)}\|\u\|_{\H}^\frac{r-3}{2(r-1)}\nonumber\\&\leq \frac{\delta_1}{1-\delta_1}\|\mathcal{F}^{1}(\u)\|_{\H}+C_{\delta_1,\delta_2,\mu,\beta}\|\u\|_{\H}, 
	\end{align}
	where  $C_{\delta_1,\delta_2,\mu,\beta}=\frac{r-3}{2(r-1)}\left(\frac{1-\delta_1}{\mu^{\frac{r-1}{2}}\beta(1-\delta_2)}\right)^{\frac{2}{r-3}}\left(\frac{r+1}{2\delta_1(r-1)}\right)^{\frac{r+1}{r-3}}$.
	Now using the estimates \eqref{3.147}-\eqref{3.149} and \eqref{3.3.6} in \eqref{3.3.33}, we deduce 
	\begin{align*}
		\|\mathcal{F}^{2}(\u)\|_{\H}&\leq\mu\delta_1\|\A\u\|_{\H}+\beta\delta_2 \|\mathcal{C}(\u)\|_{\H}+\|\B(\u)\|_{\H}+\kappa\|\u\|_{\H}\nonumber\\&\leq	
		\left[\frac{2\delta_1}{1-\delta_1}+\frac{\delta_2}{1-\delta_2}\right]\|\mathcal{F}^1(\u)\|_{\H}+(C_{\delta_1,\delta_2,\mu,\beta}+\kappa)\|\u\|_{\H}. 
	\end{align*}
	Let us choose $\delta_1$ and $\delta_2$ in such a way that $\rho=\frac{2\delta_1}{1-\delta_1}+\frac{\delta_2}{1-\delta_2}<1,$ for example, one can choose $\delta_1=\frac{1}{9}$, $\delta_2=\frac{1}{5}$, so that $\rho=\frac{1}{2}.$ Then by the well-known  perturbation theorem for nonlinear $m$-accretive operators (\cite[Chapter II, Theorem 3.5]{VB1}), we conclude that the operator $\mathcal{F}^{1}+\mathcal{F}^{2}$  with the domain $\D(\A)$  is $m$-accretive in $\H$. Since  $\mathcal{F}^{1}+\mathcal{F}^{2}=\mathcal{G}+\kappa\mathrm{I}$ , the operator  $\mathcal{G}+\kappa\mathrm{I}$ is $m$-accretive in $\H$. 

	Finally, from the abstract result \cite[Theorem 1.6 (pp. 216) and 1.8 (pp. 224), Chapter 4]{VB} or \cite[Theorem 2.5, Chapter III, pp. 140]{VB1}, there exists a unique strong solution $\u\in\W^{1,\infty}([0,T];\H)$ with  $\A\u\in\mathrm{L}^{\infty}(0,T;\H),$ such that it satisfies \eqref{holdae} for a.e. $t\in[0,T]$, it is differentiable from right in $\H$ and \eqref{holdfa} holds for all $t\in[0,T)$. 
\end{proof}

\begin{remark}
As discussed in \cite{KT2}, one can obtain the results obtained in the previous Theorem in the following way also.	Let us define 
$\v=\frac{\d\u}{\d t}$. Then differentiating \eqref{kvf} with respect to $t$, we find for a.e. $t\in[0,T]$
	\begin{align}
	\frac{\d\v}{\d t}=-\mu \A\v(t)-\B(\v(t),\u(t))-\B(\u(t),\v(t))-\beta r\widetilde{\mathcal{C}}(\u(t))\v(t)+\frac{\d\f(t)}{\d t}.
	\end{align}
  Taking inner product with $\v(\cdot)$, one gets
	\begin{align}\label{2.96}
&	\frac{1}{2}\frac{\d}{\d t}\|\v(t)\|_{\H}^2+\mu \|\nabla\v(t)\|_{\H}^2+\beta r(\widetilde{\mathcal{C}}(\u(t))\v(t),\v(t))\nonumber\\&=-(\B(\v(t),\u(t)),\v(t))+\left(\frac{\d\f(t)}{\d t},\v(t)\right).
	\end{align}
	Note that, from \eqref{Gaetu}, we have $(\widetilde{\mathcal{C}}(\u)\v,\v)=\||\u|^{\frac{r-1}{2}}|\v|\|_{\H}^2\geq 0$. A calculation similar to \eqref{2.30} gives 
	\begin{align*}
	|(\B(\v,\u),\v)|&\leq\frac{\mu}{2}\|\nabla\v\|_{\H}^2+\frac{\beta r}{2}\||\u|^{\frac{r-1}{2}}|\v|\|_{\H}^2+\frac{r-3}{\mu(r-1)}\left(\frac{2}{\beta r \mu (r-1)}\right)^{\frac{2}{r-3}}\|\v\|_{\H}^2,
	\end{align*}
	for $r> 3$.	 By using the Cauchy-Schwarz inequality, we estimate $\left|\left(\frac{\d\f(t)}{\d t},\v\right)\right|$ as 
	\begin{align*}
\left|\left(\frac{\d\f(t)}{\d t},\v\right)\right|\leq\left\|\frac{\d\f(t)}{\d t} \right\|_{\H} \|\v\|_{\H}.
	\end{align*}
    Utilizing above estimates in \eqref{2.96} to get for a.e. $t\in[0,T]$
    \begin{align}\label{w11}
    		&\frac{1}{2}\frac{\d}{\d t}\|\v(t)\|_{\H}^2+\frac{\mu}{2} \|\nabla\v(t)\|_{\H}^2+\frac{\beta r}{2}\||\u(t)|^{\frac{r-1}{2}}|\v(t)|\|_{\H}^2 \nonumber\\&\leq\left\|\frac{\d\f(t)}{\d t} \right\|_{\H}\|\v(t)\|_{\H} +\frac{r-3}{\mu(r-1)}\left(\frac{2}{\beta r \mu (r-1)}\right)^{\frac{2}{r-3}}\|\v(t)\|_{\H}^2.
      \end{align}
      Therefore, we have for all $t\in[0,T]$ 
      \begin{align}\label{eqn-fin}
      &	\|\v(t)\|_{\H}^2+\mu\int_0^t \|\nabla\v(s)\|_{\H}^2ds+\beta\int_0^t\||\u(s)|^{\frac{r-1}{2}}|\v(s)|\|_{\H}^2ds \nonumber\\&\leq\|\v(0)\|_{\H}^2+2\int_0^t\left\|\frac{\d\f(t)}{\d t} \right\|_{\H}\|\v(s)\|_{\H}\d s+2\widehat{\varrho}\int_0^t\|\v(s)\|_{\H}^2\d s,
      \end{align}
      where $\widehat{\varrho}=\frac{r-3}{\mu(r-1)}\left(\frac{2}{\beta r \mu (r-1)}\right)^{\frac{2}{r-3}}$. 
  On employing the generalized version of Gr\"onwall's inequality (Lemma \ref{lem-non-gro}), we obtain 
   \begin{align}\label{w111}
   	\sup_{t\in[0,T]}\|\v(t)\|_{\H}^2\leq\left\{\|\v(0)\|_{\H}+\int_0^T\left\|\frac{\d\f(t)}{\d t} \right\|_{\H}\d t\right\}^2e^{2{\widehat{\varrho}} T}. 
   \end{align}
   Using \eqref{w111} in \eqref{eqn-fin}, one can deduce that 
   	\begin{align}\label{299}
	&\sup_{t\in[0,T]}\|\v(t)\|_{\H}^2+\mu \int_0^T\|\nabla\v(t)\|_{\H}^2
	\d t+ \beta r\int_0^T\||\u(t)|^{\frac{r-1}{2}}|\v(t)|\|_{\H}^2\d t\nonumber\\&\leq 
	C(\beta,\mu,r,T,\|\v(0)\|_{\H},\|\f\|_{\W^{1,1}([0,T];\H)}).
\end{align}
	Note that 
	\begin{align*}
	\|\v(0)\|_{\H}&=\left\|\frac{\d\u}{\d t}(0)\right\|_{\H}= \|\mu\A\u(0)+\B(\u(0))+\beta\mathcal{C}(\u(0))+\f(0)\|_{\H}\nonumber\\&\leq\mu\|\A\u_0\|_{\H}+\||\u_0||\nabla\u_0|\|_{\H}+\beta\|\u_0\|_{\widetilde\L^{2r}}^r+\|\f(0)\|_{\H}\nonumber\\&\leq C\left(\mu,\beta,r,\|\f(0)\|_{\H},\|\A\u_0\|_{\H}\right)<
	+\infty,
	\end{align*}
whenever $\u_0\in\D(\A)$, since for $2\leq d\leq 4$, $\D(\A)\hookrightarrow\H^2(\mathcal{O})\hookrightarrow\L^p(\mathcal{O}),$ for all $p\in(1,\infty)$. Note that $\f\in\mathrm{W}^{1,1}([0,T];\H)$  implies that $\f\in\C([0,T];\H)$. In the above estimate, we have  used the Gagliardo-Nirenberg inequality also, that is, for $2\leq d\leq 4$, 
\begin{align*}
\||\u_0||\nabla\u_0|\|_{\H}\leq\|\u_0\|_{\widetilde\L^4}\|\nabla\u_0\|_{\widetilde\L^4}\leq C\|\u_0\|_{\widetilde\L^4}\|\u_0\|_{\H^2}^{\frac{4+d}{8}}\|\u_0\|_{\H}^{\frac{4-d}{8}}\leq C\|\u_0\|_{\H^2}^2. 
\end{align*}
	Taking the inner product with $\A\u(\cdot)$ to the first equation in \eqref{kvf}, we find 
\begin{align}\label{2p73}
	\mu \|\A\u(t)\|_{\H}^2&=-\left(\frac{\d\u(t)}{\d t},\A\u(t)\right) -(\B(\u(t)),\A\u(t))-\beta(\mathcal{C}(\u(t)),\A\u(t))-(\f(t),\A\u(t)),
\end{align}
for a.e. $t\in[0,T]$. We estimate the first term on the right hand side of the equality \eqref{2p73} using H\"older's, and Young's inequalities as 
\begin{align}\label{2p75}
	|(\B(\u),\A\u)|&\leq\||\u||\nabla\u|\|_{\H}\|\A\u\|_{\H}\leq\frac{\mu }{2}\|\A\u\|_{\H}^2+\frac{1}{2\mu }\||\u||\nabla\u|\|_{\H}^2. 
\end{align}
For $r>3$, we  estimate the final term from \eqref{2p75} using a calculation  similar to \eqref{rg3c} as 
\begin{align}\label{4p33}
	&	\int_{\mathcal{O}}|\u(x)|^2|\nabla\u(x)|^2\d x
	\nonumber\\&\leq\beta\mu\left(\int_{\mathcal{O}}|\u(x)|^{r-1}|\nabla\u(x)|^2\d x\right)+2\varrho\mu\left(\int_{\mathcal{O}}|\nabla\u(x)|^2\d x \right).
\end{align}
Therefore, substituting \eqref{3}, \eqref{2p75}-\eqref{4p33} in \eqref{2p73},  we obtain 
\begin{align}
	&\frac{\mu}{2}\|\A\u\|_{\H}^2+\frac{\beta}{2}\||\u|^{\frac{r-1}{2}}\nabla\u\|_{\H}^2\nonumber\\&\leq\|\f\|_{\H}\|\A\u\|_{\H}+
	\left\|\frac{\d\u}{\d t}\right\|_{\H}\|\A\u\|_{\H}+ 2\varrho\mu\|\nabla\u\|_{\H}^2\nonumber\\&\leq\frac{\mu}{4}\|\A\u\|_{\H}^2+\frac{1}{\mu}
	\left(\|\f\|_{\H}+\left\|\frac{\d\u}{\d t}\right\|_{\H}\right)^2 +2\varrho\mu\|\nabla\u\|_{\H}^2.  
\end{align}

Thus, finally using the fact that $\u\in\C([0,T];\V)$, $\frac{\d\u}{\d t} \in\mathrm{L}^{\infty}(0,T;\H)$, $\f\in\mathrm{W}^{1,1}([0,T];\H)$, we deduce that $\u\in\mathrm{L}^{\infty}(0,T;\D(\A))$, for any $\u_0\in\D(\A)$. The case of $r=3$ and $2\beta\mu\geq 1$ can be handled in a similar way. 
\end{remark}

 \medskip\noindent
{\bf Acknowledgments:} With great appreciation, the authors would like to thank the anonymous referee for his or her insightful remarks and ideas, which greatly enhanced the quality of the manuscript.   M. T. Mohan would  like to thank the Department of Science and Technology (DST), India for Innovation in Science Pursuit for Inspired Research (INSPIRE) Faculty Award (IFA17-MA110) and Indian Institute of Technology Roorkee, for providing stimulating scientific environment and resources. The author would also like to thank Prof. J. C. Robinson, University of Warwick for useful discussions and providing the crucial reference \cite{CLF}.  The authors would like to thank Dr. Kush Kinra, NOVA University Lisbon, Portugal for useful discussions.

\end{document}